\documentclass[centering,11pt,reqno]{amsart}  

\usepackage[utf8]{inputenc}
\usepackage[T1]{fontenc}
\usepackage{amsmath,amsthm}
\usepackage{amsfonts,amssymb}
\usepackage{paralist}
\usepackage{xcolor}
\usepackage[colorlinks,cite color=blue,pagebackref=true,pdftex]{hyperref}
\usepackage[margin=2cm]{geometry}

\usepackage{float}
\usepackage{wrapfig}

\usepackage{graphicx}

\usepackage{cleveref}

\newtheorem{thm}{Theorem}[section]
\newtheorem{lem}[thm]{Lemma}
\newtheorem{cor}[thm]{Corollary}
\newtheorem{prop}[thm]{Proposition}

\makeatletter
\newtheorem*{rep@theorem}{\rep@title}
\newcommand{\newreptheorem}[2]{%
\newenvironment{rep#1}[1]{%
\hypersetup{linkcolor=black}%
 \def\rep@title{#2~\ref{##1}}%
 \begin{rep@theorem}}%
 {\end{rep@theorem}}}
\makeatother

\theoremstyle{definition}
\newtheorem{defn}[thm]{Definition}
\newtheorem{remark}[thm]{Remark}

\newreptheorem{theorem}{Theorem}

\newcommand{\Def}[1]{\textbf{#1}} 

\newcommand{\1}{\mathbf{1}}%
\renewcommand{\emptyset}{\varnothing}
\newcommand{\eps}{\varepsilon}
\newcommand{\R}{\mathbb{R}}

\newcommand\wt{w}%

\newcommand\colpat{\widetilde{\arb}}%

\renewcommand\AA{\mathbb{A}}%
\newcommand\SymGrp{\mathfrak{S}}%

\newcommand\PP{\Pi}%
\newcommand\PPGKZ{\psi}%

\renewcommand\L{\mathcal{L}}%

\DeclareMathOperator{\argmax}{argmax}
\DeclareMathOperator{\argmin}{argmin}
\DeclareMathOperator{\conv}{conv}
\DeclareMathOperator{\aff}{aff}

\newcommand\PS[1][m-1,n-1]{P_{#1}}%
\newcommand\PSV{V_{m,n}}%
\newcommand\PSVo{V^\circ_{m,n}}%

\date{\today}

\keywords{linear programming, geometry of pivot rules, particle collisions,
associahedra, multiplihedra, constrainahedra}

\subjclass[2020]{%
90C05, %
90C57, %
52B12, %
52B11} %

\begin{document}

\title{From linear programming to colliding particles}

\author[A.~Black \and N.~L\"utjeharms \and R.~Sanyal]{Alexander
E. Black \and Niklas L\"utjeharms \and Raman Sanyal}

\address[A.~Black]{Dept.\ Mathematics, Univ. of California,
Davis, CA 95616, USA}
\email{aeblack@ucdavis.edu}

\address[N.~L\"utjeharms, R.~Sanyal]{Institut f\"ur Mathematik,
Goethe-Universit\"at Frankfurt, Frankfurt am Main, Germany} 
\email{sanyal@math.uni-frankfurt.de}

\begin{abstract}
    Although simplices are trivial from a linear optimization standpoint, the
    simplex algorithm can exhibit quite complex behavior. In this paper we
    study the behavior of max-slope pivot rules on (products of) simplices and
    describe the associated pivot rule polytopes. For simplices, the pivot
    rule polytopes are combinatorially isomorphic to associahedra. To prove
    this correspondence, we interpret max-slope pivot rules in terms of the
    combinatorics of colliding particles on a line. For prisms over simplices,
    we recover Stasheff's multiplihedra. For products of two simplices we get
    new realizations of constrainahedra, that capture the combinatorics of
    certain particle systems in the plane.
\end{abstract}

\maketitle

\newcommand\inner[1]{\langle {#1} \rangle}%
\newcommand\arb{\mathcal{A}}%
\newcommand\vopt{v_{\mathrm{opt}}}%
\newcommand\Ass[1][n-2]{\mathrm{Asso}_{#1}}%

\section{Introduction}\label{sec:intro}

A linear program (LP) is an optimization problem of the form
\[
    \begin{array}{lr@{\,+\,}c@{\,+\,}r@{\,}l}
    \text{maximize}  &  c_1x_1 & \cdots & c_nx_n  \\
    \text{subject to}  &  a_{i1} x_1 & \cdots & a_{in} x_n & \ \le \ b_i \quad
    \text{ for } i=1,\dots,N \, .
\end{array}
\]
Geometrically, the set $P \subset \R^n$ of feasible solutions is a polyhedron
and a generic objective function $c = (c_1,\dots,c_n)$ induces an
orientation on the geometric graph $G(P)$. The orientation is acyclic
with a unique sink at the optimal vertex $\vopt$. The simplex algorithm starts
at a given vertex $v$ of $P$ and proceeds along directed edges to $\vopt$.
The choice which edges to pursue is governed by a pivot rule.

A polytope $P \subset \R^n$ is a \emph{simplex} if its vertices are affinely
independent. Simplices are trivial from an optimization viewpoint as any two
vertices of $P$ are adjacent. However, sophisticated algorithms such as the
simplex algorithm can exhibit complex and interesting behavior on trivial
instances. In this paper, we describe a beautiful and unexpected connection
between the behavior of certain pivot rules on (products of) simplices and the
combinatorics of colliding particles.

A pivot rule is memory-less if its behavior on a linear program $(P,c)$ is
captured by an arborescence (or rooted tree) on $G(P)$. In~\cite{PivPoly} we
studied families of memory-less pivot rules that are parametrized by weight
vectors $\wt$. We showed that for any linear program $(P,c)$ there is a
polytope $\PP(P,c)$, the \emph{pivot rule polytope}, whose vertices correspond
to the arborescences of $(P,c)$ induced by the family of pivot rules. The
facial structure of $\PP(P,c)$ reflects the relation between the different
rules on $(P,c)$.  For the \emph{max-slope} pivot rule, a generalization of
the shadow vertex algorithm of Gass and Saaty~\cite{GassSaaty} that we
introduced in~\cite{PivPoly}, we were surprised to observe that the numbers of
max-slope arborescences for simplices are given by the Catalan numbers.  Much
of the combinatorics surrounding the Catalan numbers is famously embodied by
the \Def{associahedron} $\Ass[n-2]$, a certain partially ordered set that is
ubiquitous in geometric and algebraic combinatorics. Lee~\cite{lee} showed
that $\Ass[n-2]$ is isomorphic to the face poset of a simple $(n-2)$-dimensional
polytope. Since then many different ways of realizing $\Ass$ as a polytope
have been discovered in various contexts; see, for example,
\cite{CeballosSantosZiegler}. The present paper adds a new construction to the
list.

\begin{reptheorem}{thm:main_asso}
    Let $P$ be an $(n-1)$-dimensional simplex and $c$ a generic objective
    function. Then the max-slope pivot rule polytope $\PP(P,c)$ is
    combinatorially isomorphic to the $(n-2)$-dimensional associahedron
    $\Ass$.
\end{reptheorem}

Associahedra describe the combinatorics of colliding particles. Consider $n$
distinct and ordered particles on the real line.  Particles move, collide, and
merge until there is a single particle left. The various collisions can be
time-independently recorded by a \emph{bracketing}. For example $(12)(345)$
states that at some point particles $1$ and $2$ collide and, earlier or later,
$3,4$, and $5$ collide simultaneously. Eventually, the two remaining particles
collide. The associahedron $\Ass$ is the set of bracketings of $123\dots n$
partially ordered by refinement. In order to prove \Cref{thm:main_asso}, we
describe a geometric correspondence between max-slope arborescences and
bracketings. For that, the objective function $c$ gives rise to velocities for
the $n$ particles and the weights $\wt$ give each particle a location at time
$t=0$. For $t >0$, the particles start to move from their locations at
constant velocity.  If two particles collide, the slower particle is absorbed
by the faster one, which continues at its original velocity.  For $t \gg 0$,
only particle $n$ is left. We record the particle $\arb(i)$ that absorbs the
particle $i$.  Towards a proof of \Cref{thm:main_asso}, we show that these
maps $\arb$, called collision patterns, are in bijection with bracketings and
are precisely the max-slope arborescences of an $(n-1)$-simplex with objective
function $c$.

Bottman and Poliakova~\cite{BottmanPoliakova} studied a more general setup for
particle collisions. For $m,n \ge 1$, they consider $m \cdot n$ particles
sitting at the intersections of $m$ horizontal and $n$ vertical lines in the
plane.  The particles are allowed to move horizontally or vertically but they
must retain their colinearities. The collisions can be recorded by
\emph{rectangular brackets} or, equivalently, by a partially ordered set on
the (spaces between) the horizontal and vertical lines. The resulting poset of
rectangular bracketings is called the \Def{constrainahedron} $C(m,n)$. For
$m=1$, this is the associahedron. For $m=2$, $C(2,n)$ is isomorphic to the
\emph{multiplihedron}, a poset first described by
Stasheff~\cite{Stasheff-book} and realized as a polytope by
Forcey~\cite{Forcey}. It is shown in~\cite{BottmanPoliakova}, that $C(m,n)$ is
the face poset of a generalized permutahedron~\cite{GenPermOrig}. Chapoton and
Pilaud~\cite{ChapotonPilaud} introduced a remarkable operation on products of
generalized permutahedra, called \emph{shuffle products}, and gave a different
realization of $C(m,n)$ as the shuffle product of Loday associahedra.  We give
new realizations of constrainahedra that, in particular, are not generalized
permutahedra.

\begin{reptheorem}{thm:constrainahedron}
    Let $\PS$ be the product of an $(m-1)$-simplex and an $(n-1)$-simplex and
    let $c$ be a generic objective function. Then the max-slope pivot rule
    polytope $\PP(\PS,c)$ is combinatorially isomorphic the
    $(m,n)$-constrainahedron $C(m,n)$.
\end{reptheorem}

The combinatorial construction of constrainahedra was motivated by questions in
homotopical algebra~\cite{poliakova-thesis} as well as Gromov compactifications
of configuration spaces. More precisely, Bottman~\cite{bottman} constructed
\emph{$2$-associahedra} as posets capturing the behavior of ordered particles on
parallel lines in the plane without colinearities. It is conjectured that
$2$-associahedra are face posets of convex polytopes and we hope that our
results provide a new point of view on this conjecture.

Our techniques generalize  to higher products of simplices. We focus on the
case in which all but one simplex is $1$-dimensional.

\begin{reptheorem}{thm:multi_multiplihedron}
    Let $(\AA,\cdot)$ be a non-associative monoid and let $f_1,\dots, f_k :
    \AA \to \AA$ be morphisms. The vertices of the max-slope pivot polytope
    of an $(n-1)$-simplex times a $k$-cube are in bijection with the possible
    ways of evaluating 
    \[
        (f_{\sigma(1)} \circ f_{\sigma(2)} \circ \cdots \circ
        f_{\sigma(k)})(a_1 \cdot a_2 \cdots a_n)\, ,
    \]
    where $a_1,\dots,a_n \in \AA$ and $\sigma \in \SymGrp_k$ is a permutation.
\end{reptheorem}

These are the vertices of the $(m,n)$-multiplihedra of Chapoton and
Pilaud~\cite{ChapotonPilaud}; see also~\cite{Hochschild}. In light of these
results, we conjectured together with Vincent Pilaud that max-slope pivot rule
polytopes of higher products of simplices are isomorphic to shuffle products of
associahedra. This was proven by Germain Poullot in his
thesis~\cite{poullot:tel-04269354} and will appear in a forthcoming paper. We
close with some results and conjectures on the enumerative combinatorics on the
vertex numbers of these $(m,n)$-multiplihedra. 

We collect the necessary motivation and background on max-slope pivot
rules and pivot rule polytopes in \Cref{sec:max-slope}. In
\Cref{sec:simplices}, we develop the combinatorics of max-slope arborescences
on simplices. In \Cref{sec:particles} we describe the correspondence to
colliding particles and prove \Cref{thm:main_asso}. In \Cref{sec:products}, we
discuss products of two simplices and prove \Cref{thm:constrainahedron}. In
\Cref{sec:higher} we discuss higher products of simplices and focus on the
simplex-times-cube case. 

\subsection*{Acknowledgements} We would like to thank Aenne Benjes, Jes\'{u}s
De Loera, Vincent Pilaud, Dasha Poliakova, Alex Postnikov, Germain Poullot,
and Hugh Thomas for insightful conversations. The last author would like to
thank the organizers of the workshop \emph{Combinatorics and Geometry of
Convex Polyhedra}, held at the Simons Center for Geometry and Physics, where
he first learned about constrainahedra. The first author was supported by the NSF GRFP and NSF DMS-1818969.

\section{Max-slope pivot rule polytopes}\label{sec:max-slope}

We recall the necessary background on max-slope pivot rules and the associated
pivot rule polytopes. We refer the reader to~\cite{PivPoly} for details and
proofs.

Let $P \subset \R^n$ be a convex polytope with vertex set $V(P) =
\{v_1,\dots,v_n\}$. We write $E(P)$ for the edges of $P$.  We call $c \in
\R^n$ \Def{edge-generic} with respect to $P$ if $\inner{c,u} \neq \inner{c,v}$
for all edges $uv \in E$. The vector $c$ defines an objective function $x
\mapsto \inner{c,x}$ and $(P,c)$ is a \Def{linear program} (LP).

We denote the unique maximizer of $c$ over $P$ by $\vopt$.  A
\Def{$c$-arborescence} is a map $\arb : V(P) \setminus \vopt \to V(P)$ such
that $\arb(v)$ is a $c$-improving neighbor of $v \neq \vopt$. If $c$ is clear
from the context, we simply call $\arb$ an arborescence. Arborescences encode
the behavior of memory-less pivot rules on the linear program $(P,c)$: From a
starting vertex $v \in V(P)$, the simplex algorithm constructs a monotonically
increasing path $v = v_0v_1\dots v_m = \vopt$ in the graph of $P$ that
satisfies $v_i = \arb(v_{i-1})$ for all $i=1,\dots,m$.

The max-slope pivot rule introduced in~\cite{PivPoly} generalizes the
well-known shadow-vertex simplex algorithm.  For $\wt \in \R^d$ generic and
linearly independent of $c$, we define the \Def{max-slope} arborescence
$\arb^w$ of $(P,c)$ by\footnote{As the notation suggests, $\argmax$ returns
the improving neighbor of $v$ that maximizes the given quantity.}
\begin{equation}\label{eqn:max_slope}
    \arb^w(v) \ := \ \argmax \left\{ 
        \frac{\inner{\wt,u - v}}{\inner{c,u - v}} : uv \in E(P),  \inner{c,u} >
        \inner{c,v} \right\} \, .
\end{equation}
A geometric interpretation for a max-slope arborescence can be given as
follows. Define the linear projection $\pi : \R^n \to \R^2$ by $\pi(x) =
(\inner{c,x},\inner{\wt,x})$. The projection of an edge $uv \in E$ to the
plane has a well-defined slope with respect to the $x$-axis and $\arb^\wt(v)$
selects the edge with maximal slope.

For a $c$-arborescence $\arb$ we define
\begin{equation}\label{eqn:GKZ}
    \PPGKZ(\arb) \ := \ \sum_{v \neq \vopt} \frac{\arb(v) - v}{\inner{c,\arb(v)
    - v}} \, .
\end{equation}
and the \Def{max-slope pivot rule polytope} of $(P,c)$ 
\[
    \PP(P,c) \ := \ \conv \big\{ \PPGKZ(\arb) : \arb \text{ $c$-arborescence
    of } (P,c) \big\} \, .
\]

This is a convex polytope, that geometrically encodes the various max-slope
arborescences of $(P,c)$. The \Def{support function} $h_Q : \R^n \to \R$
of a polytope $Q \subset \R^n$ is defined by $h_Q(\wt) := \max \{
    \inner{\wt,x} : x \in Q \}$. For $\wt \in \R^n$ we write $Q^\wt$ for the
face $Q \cap \{ x : \inner{\wt,x} = h_Q(\wt) \}$.

\begin{thm}[{\cite[Theorem~1.4]{PivPoly}}]\label{thm:PP}
    Let $P$ be a polytope and $c$ an edge-generic objective function. Then for
    every generic $\wt$
    \[
        \PP(P,c)^\wt \ = \ \{ \PPGKZ(\arb^w) \} \, .
    \]
    In particular, the vertices of $\PP(P,c)$ are in bijection with max-slope
    arborescences of $(P,c)$.
\end{thm}

The affine hull $\aff(P)$ of $P$ is the inclusion-minimal affine subspace
containing $P$.

\begin{prop}\label{prop:PP_aff}
    Let $P \subset \R^n$ be a polytope and $c \in \R^n$ an edge-generic
    objective function. Then
    \[
        \aff(\PP(P,c)) \ = \ (-q + \aff(P)) \cap 
                        \{ x : \inner{c,x} = |V(P)| - 1 \} \, ,
    \]
    for any $q \in P$. In particular, $\dim \PP(P,c) = \dim P - 1$.
\end{prop}
\begin{proof}
    Let $L$ denote the right-hand side of the given equation.  It follows
    from~\eqref{eqn:GKZ} that $\PPGKZ(\arb) \in L$ for every arborescence
    $\arb$ and hence $\aff(\PP(P,c)) \subseteq L$. Since $c$ is not constant
    on $P$, it follows that $L$ is of dimension $\dim P - 1$. Theorem
    1.4 in~\cite{PivPoly} shows that $\dim \PP(P,c) = \dim P - 1$, which
    finishes the proof.
\end{proof}

We record the following properties of max-slope pivot rule polytopes that
directly follow from \eqref{eqn:max_slope} and \eqref{eqn:GKZ}.
\begin{cor} \label{cor:PP_inv}
    Let $P \subset \R^n$ be a polytope and $c$ an edge-generic objective
    function.
    \begin{enumerate}[\rm (i)]
        \item Let $\wt \in \R^n$ generic and $\alpha \in \R$. Then $\arb^{\wt
            + \alpha c} = \arb^\wt$.
        \item If $h \in \R^n$ such that $x \mapsto \inner{h,x}$ is constant on
            $P$, then $\PP(P, c + h) = \PP(P,c)$.
        \item For any $b \in \R^n$, $\PP(P + b, c) = \PP(P,c)$.
        \item If $T$ is an invertible linear transformation, then
            $\Pi(TP,(T^{-1})^tc)  \ = \ T\Pi(P,c)$.
    \end{enumerate}
\end{cor}

\newcommand\Simp[1][n-1]{\Delta_{#1}}%
\section{Max-slope pivot rules on simplices}\label{sec:simplices}

Let $P$ be an $(n-1)$-dimensional simplex and $c$ an edge-generic objective
function. Since any two simplices of the same dimension are affinely
isomorphic, we can appeal to \Cref{cor:PP_inv} and assume that the simplex
$P$ is the $(n-1)$-dimensional standard simplex 
\[
    \Simp \ := \ \conv( e_1, \dots, e_n )  \ \subset \ \R^n \, .
\]
Furthermore, up to relabelling vertices, we can assume that the evaluations $c_i
= \inner{c,e_i}$ for $i=1,\dots,n$ satisfy $c_1 < c_2 < \cdots < c_n$.

For $n \ge 1$, we write $[n] := \{1,2,\dots,n\}$. Since $\Simp$ has a complete
graph, every vertex $e_j$ with $j > i$ is a $c$-improving neighbor of $e_i$.
Thus, arborescences of $(\Simp,c)$ bijectively correspond to maps $\arb :
[n-1] \to [n]$ such that $\arb(i) > i$ for all $i$. At times, we will abuse
notation and write $(i,j) \in \arb$ if $\arb(i) = j$. Since every vertex can
choose an improving neighbor independently, there are exactly
$(n-1)!$ arborescences of $(\Simp,c)$. For the tetrahedron, it turns out that
of the six arborescences only five are max-slope arborescences and
\Cref{fig:tetra} prompts the following definition.

\begin{figure}[h]
    \begin{center}
        \includegraphics[width=12cm]{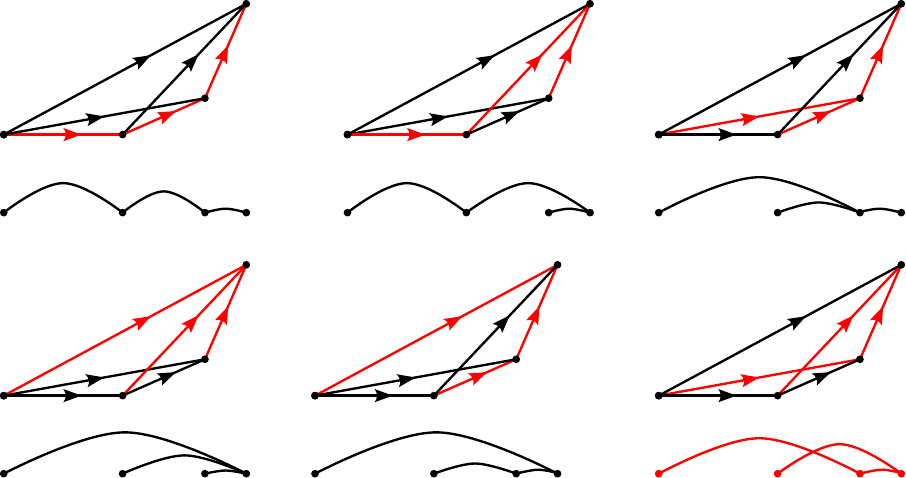}
    \end{center}
    \caption{Tetrahedron with orientations induced by a generic objective
    function. The six possible $c$-arborescences are shown in red together
    with a simpler depiction as maps $\arb : [3] \to [4]$ below. The
    bottom-right arborescence is not max-slope.}
    \label{fig:tetra}
\end{figure}

\begin{defn}[Noncrossing arborescence]
    An arborescence $\arb : [n-1] \to [n]$ is \Def{noncrossing} if  $\arb(j)
    \le \arb(i)$ for all $1 \le i < j < n$ with $j < \arb(i)$.
\end{defn}

In other words, $\arb$ is noncrossing if there are no $a < i < b < j$ such
that $\arb(a)  = b$ and $\arb(i) = j$:
\begin{center}
        \includegraphics[height=1cm]{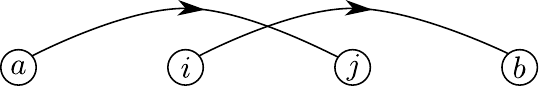}
\end{center}

\begin{thm}\label{thm:noncross}
    Let $\arb : [n-1] \to [n]$ be an arborescence. Then $\arb$ is a max-slope
    arborescence for $(\Simp,c)$ if and only if $\arb$ is noncrossing.
\end{thm}

\newcommand\arbA{\arb_1}%
\newcommand\arbAA{\arb_2}%
For the proof, we use a canonical decomposition of noncrossing arborescences;
see \Cref{fig:noncross}.

\begin{figure}[h]
    \begin{center}
    \includegraphics[width=12cm]{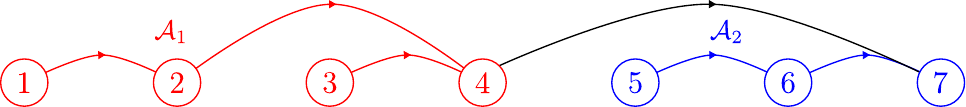}
    \end{center}
    \caption{A noncrossing arborescence on $7$ nodes. $r=4$ is minimal with
    $\arb(r) = 7$. The decomposition into $\arbA,\arbAA$ is obtained by
    restricting to $1,\dots,4$ and $5,\dots,7$.  The nodes $1,3,5$ are
    leaves and they are all immediate leaves.}
    \label{fig:noncross}
\end{figure}

\begin{lem}\label{lem:decomp}
    Let $\arb : [n-1] \to [n]$ be a noncrossing arborescence on $n \ge 2$
    nodes.  Let $r$ be minimal with $\arb(r) = n$. Then $\arbA : [r-1] \to
    [r]$ given by $\arbA(i) := \arb(i)$ for $i < r$ and $\arbAA : [n-r-1] \to
    [n-r]$ given by $\arbAA(i) := \arb(r+i) - r$ are noncrossing
    arborescences.  Moreover, $\arb$ is uniquely determined by
    $(\arbA,\arbAA)$.
\end{lem}

\Cref{lem:decomp} allows us to count noncrossing arborescences. Recall that
the famous \Def{Catalan numbers} $C_k$ are given by $C_0 := 1$ and \begin{equation}\label{eqn:catalan}
    C_{k+1} \ := \ C_0C_k + C_1 C_{k-1} + \cdots + C_{k-1}C_1 + C_kC_0 
\end{equation}
for $k \ge 0$; see also~\cite{catalan}. \Cref{lem:decomp} shows that 
noncrossing arborescence $\arb$ on $n$ nodes are in bijection with pairs of
noncrossing arborescences $(\arb_1,\arb_2)$ of sizes $r$ and $n-r$,
respectively, for $r = 1,\dots,n-1$. We will make the connection to Catalan combinatorics more explicit
when we describe the combinatorial structure of $\PP(\Simp,c)$. For now we record:

\begin{cor}\label{cor:Catalan}
    There are exactly $C_{n}$ many noncrossing arborescences on $n+1 \ge 1$
    nodes.
\end{cor}

In light of \Cref{thm:noncross} and the previous result, this shows that there
are $C_{n} = \frac{1}{n+1} \binom{2n}{n} \sim 4^n/\sqrt{n^3 \pi}$ many
max-slope arborescences of an $n$-dimensional simplex, which has $n+1$
vertices and $\frac{n^2+n}{2}$ edges.

We call $k < n$ a \Def{leaf} of $\arb$ if there is no $i$ with $\arb(i) = k$.

\begin{lem}\label{lem:immediate}
    Let $\arb$ be a noncrossing arborescence on $n \ge 2$ nodes. Then there is
    a leaf $k$ with $\arb(k) = k+1$.
\end{lem}

We call such a leaf an \Def{immediate leaf} of $\arb$.

\begin{proof}
    If $n = 2$, then $\arb = \{ (1,2) \}$ and $1$ is an immediate leaf.  For
    $n \ge 3$, let $(\arbA,\arbAA)$ be the decomposition of $\arb$ of
    \Cref{lem:decomp}.  By induction $\arbA$ or $\arbAA$ has an immediate
    leaf.
\end{proof}

\newcommand\C[1]{\tau({#1})}%
For a fixed $\wt \in \R^n$, we define the \Def{slope} of $1 \le i < j \le n$ by
\[
    \C{i,j} \ := \ \frac{\wt_j - \wt_i}{c_j - c_i} \, .
\]
Then $\arb : [n-1] \to [n]$ is a max-slope arborescence with respect to $\wt$ if
and only if for all $1 \le i < n$ 
\begin{equation}\label{eqn:coh}
    \C{i,\arb(i)} \ > \ \C{i,k} \quad \text{ for all } k > i \text{ and } k
    \neq \arb(i) \,
    .
\end{equation}

We will need the following convexity result.

\begin{lem}\label{lem:conv}
    For $1 \le r < s < t \le n$ we have 
    \[
        \C{r,t}  >  \C{r,s} \ \Longleftrightarrow \
        \C{s,t}  >  \C{r,t} 
        \quad \text{ and } \quad
        \C{r,t}  <  \C{r,s} \ \Longleftrightarrow \
        \C{s,t}  <  \C{r,t}.
    \]
\end{lem}
\begin{proof}
    Note that 
    \[
        \C{r,t} \ = \ \frac{c_s-c_r}{c_t-c_r} \C{r,s} +
        \frac{c_t-c_s}{c_t-c_r} \C{s,t} \, ,
    \]
    which implies $\max(\C{r,s},\C{s,t}) \ge \C{r,t} \ge 
    \min(\C{r,s},\C{s,t})$. If $\C{r,s} \neq \C{s,t}$, then the inequalities
    are strict.
\end{proof}

\begin{proof}[Proof of \Cref{thm:noncross}]
    Let $\wt \in \R^n$ be a weight such that $\arb$ is the max-slope
    arborescence with respect to $\wt$. Assume that there are $1 \le a < i < b
    < j \le n$ such that $(a,b)$ and $(i,j)$ are crossing arcs of $\arb$. We
    claim that the following inequalities hold
    \[
        \C{i,b} \ < \ \C{i,j} \ < \ \C{b,j} \ < \ \C{a,j} \ < \ \C{a,b}  \, .
    \]
    The first and last inequality follow from~\eqref{eqn:coh} and the fact
    that $\arb(i)=j$ and $\arb(a)=b$. The second and third inequality follow
    from \Cref{lem:conv} for $i < b < j$ and $a < b < j$, respectively.
    In particular, we have $\C{a,b} > \C{i,b}$. On the other hand $\arb(a) = b$
    implies $\C{a,b} > \C{a,i}$ and \Cref{lem:conv} gives $\C{i,b} >
    \C{a,b}$. This is a contradiction.

    Now let $\arb$ be a noncrossing arborescence on $n$ nodes and recall that
    $c_1 < \cdots < c_{n}$. We show the existence of a suitable $\wt =
    (\wt_1,\dots,\wt_{n})$ by induction on the number of nodes $n$.  Using
    \eqref{eqn:coh}, we need to verify that for every $1 \le i < k \le n$ with
    $k \neq \arb(i)$
    \[
        0 \ < 
            (c_k - c_i) \wt_{\arb(i)}
          - (c_{\arb(i)} - c_i) \wt_k
          + (c_{\arb(i)} - c_k) \wt_i  \ =: \ L_{i,k}(\wt)  \, .
    \]
    For $n \le 2$, these conditions are vacuous. For $n > 2$, let $l$ be an
    immediate leaf, whose existence is guaranteed by \Cref{lem:immediate}. We
    claim that the coefficient of $\wt_l$ in every inequality $L_{i,k}(\wt) > 0$
    is nonpositive.  Since $l$ is leaf, there is no $i$ with $\arb(i) = l$ and
    $\wt_l$ only occurs in inequalities $L_{i,k}(\wt) > 0$ with $i = l$ or $k =
    l$. For $i < k = l$, the claim is clearly true. For $l = i < k$, the fact
    that $l$ is an immediate leaf and hence $\arb(l) = l+1$ shows that $k > l+1$
    and thus $c_{\arb(l)} - c_k < 0$.  This means that the inequalities that
    involve $\wt_l$ only give upper bounds on $\wt_l$ in terms of the other
    $\wt_i$. The remaining inequalities determine the coherence of the
    arborescence $\overline{\arb}$ obtained by restricting $\arb$ to $[n] \setminus l$.
    Since $\overline{\arb}$ is noncrossing on $n-1$ nodes, this system of strict
    inequalities is feasible by induction and we find a suitable $\wt_l$
    satisfying the inequalities involving $l$.
\end{proof}

\Cref{lem:conv} gives a simple criterion for finding an immediate leaf given a
generic $\wt$. Define the maximal slope at $1 \le i < n$ by $\tau(i) := \max
\{ \tau(i,j) : j > i \}$.
\begin{cor}\label{cor:imm_leaf}
    Let $\wt \in \R^n$ and $\arb^\wt$ the max-slope arborescence for
    $(\Simp,c)$. If $\tau(i) \ge \tau(k)$ for all $k$, then $i$ is an
    immediate leaf of $\arb^\wt$. 
\end{cor}
\begin{proof}
    Assume that $\arb(i) = j > i+1$, then $\tau(i,j) > \tau(i,i+1)$ and
    \Cref{lem:conv} implies $\tau(i+1,j) > \tau(i,j)$, which contradicts the
    maximality of $\tau(i)$. Hence $j = i+1$. Assume that there is $h < i$
    with $\arb^\wt(h) = i$. Then $\tau(h,i) > \tau(h,i+1)$ and again by
    \Cref{lem:conv} $\tau(h,i+1) > \tau(i,i+1) = \tau(i)$, which again
    contradicts maximality.
\end{proof}

\section{Particles with locations and velocities}\label{sec:particles}

Consider $n \ge 2$ labelled particles on the real line. Every particle $i =
1,2,\dots,n$ has a constant velocity $-c_i < 0$. We assume that the
velocities are distinct and, up to relabelling particles, satisfy $0 < c_1 <
c_2 < \cdots < c_n$. The velocities are fixed throughout.  Assume further that
at time $t = 0$, the particles are at locations $-\wt_1 \le -\wt_2 \le \cdots
\le -\wt_n$. \Cref{fig:particles} illustrates the setup.

\begin{figure}[h]
    \begin{center}
        \includegraphics[width=.95\textwidth]{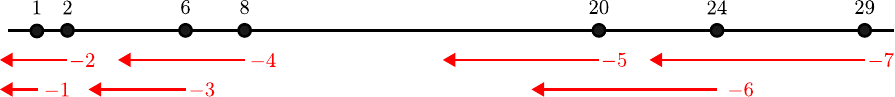}
    \end{center}
    \caption{Locations (in black) $-\wt = (1,2,6,8,20,24,29)$ and velocities
    (in red) \mbox{$-c = (-1,-2,-3,-4,-5,-6,-7)$} that produce the collision
    pattern  of \Cref{fig:noncross}.}
\label{fig:particles}
\end{figure}

Once the particles start moving from their initial locations $-\wt =
(-\wt_1,\dots,-\wt_n)$ with velocities $-c = (-c_1,\dots,-c_n)$, they will
eventually collide and merge. If particles $i < j$ collide, then particle $i$
is absorbed  by the faster particle $j$, which continues at velocity $-c_j$.
For $t \gg 0$, the only remaining particle is $n$; see \Cref{fig:evo}.

\begin{figure}[h]
    \begin{center}
        \includegraphics[width=.95\textwidth]{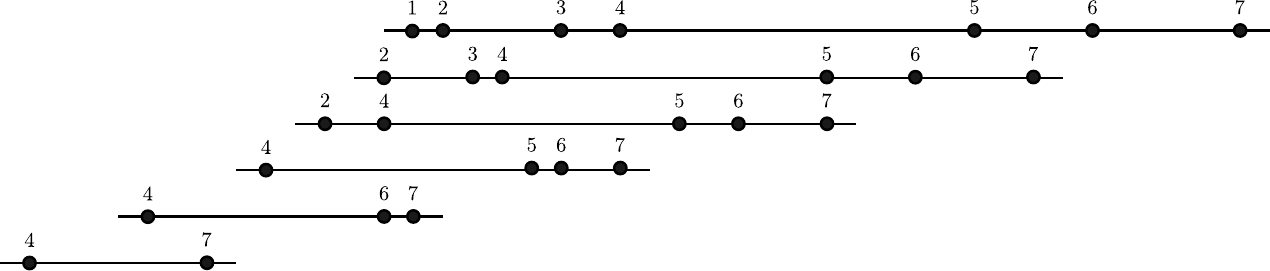}
    \end{center}
    \caption{Shows the evolution of the particles of \Cref{fig:particles} over
    time.}
    \label{fig:evo}
\end{figure}

For now we assume that the locations $-\wt$ are chosen generically, so that at
most two particles collide at any given point in time.  We record the
collisions by a map $\colpat^{-\wt} : [n-1] \to [n]$ that we call a
\Def{collision pattern}: If particle $i$ gets absorbed by particle $j$ for the
initial locations $-\wt$, then we set $\colpat^{-\wt}(i) := j$.  The connection
to max-slope arborescences of simplices is as follows.

\begin{thm}\label{thm:cs_max_arb}
    Let $n \ge 2$ and $c \in \R^n$ with $0 < c_1 < \cdots < c_n$.  For $\arb :
    [n-1] \to [n]$ and $\wt \in \R^n$, the following are equivalent.
    \begin{enumerate}[\rm i)]
        \item $\arb = \colpat^{-(\wt - \alpha c)}$ is a collision pattern for
            $n$ particles with velocities $-c$ and locations $-(\wt - \alpha
            c)$ for some $\alpha \ge 0$.
        \item $\arb = \arb^\wt$ is a max-slope arborescence of $(\Simp,c)$
            with respect to weight $\wt$.
    \end{enumerate}
\end{thm}
\begin{proof}
    Let us first assume that $-\wt_1 < \cdots < -\wt_n$.  If we fix $j > i$
    and disregard all other particles for the moment, then the time $t_{ij}$
    of collision of $i$ and $j$ satisfies
    \begin{equation}\label{eqn:tij}
        -\wt_i - t_{ij} c_i \ = \ -\wt_j - t_{ij} c_j \qquad
        \Longleftrightarrow \qquad t_{ij} \ = \ \frac{-(\wt_j - \wt_i)}{c_j -
        c_i} \, . %
    \end{equation}
    By construction, $i$ will be absorbed by particle $j$ if $t_{ij}$ is
    minimal among all $t_{ik}$ with $k > i$. Thus, for $i < n$, we observe
    \begin{equation}\label{eqn:cs_max_arb}
        \colpat^{-\wt}(i) \ = \
        \argmin \left\{ 
        \frac{-(\wt_{j} - \wt_i)}{c_{j} - c_i} : j > i \right \}
        \ = \ 
        \argmax \left\{ 
        \frac{\wt_{j} - \wt_i}{c_{j} - c_i} : j > i \right \} \ = \
        \arb^\wt(i)
        \, ,
    \end{equation}
    where the last equation is the definition of $\arb^\wt$
    in~\eqref{eqn:max_slope}. This proves the equivalence for strictly
    increasing $-\wt$. Note that both sides of \eqref{eqn:cs_max_arb} are
    invariant under replacing $\wt$ by $\wt - \alpha c$ for any $\alpha \in
    \R$.  For $\alpha \ge 0$ sufficiently large, $-(\wt - \alpha c)$ is
    strictly increasing and thus satisfies our conditions on particle
    locations. 
\end{proof}

The proof of \Cref{thm:cs_max_arb} emphasizes that for any $\wt \in \R^n$,
there is an $\alpha \ge 0$ such that $-(\wt - \alpha c)$ is strictly
increasing and the associated collision pattern is independent of the choice of
$\alpha$.  Therefore, we will use noncrossing arborescence and collision plan
interchangeably and will exclusively use the notation $\arb^\wt$. 

In the language of collision patterns, we can also interpret the decomposition of
\Cref{lem:decomp}. If $\arb^\wt$ is the collision pattern obtained from locations
$-\wt$, then there is a time $t_1$ at which there are only two particles left.
One of these two particles is clearly $n$, the other is the last particle that
is absorbed by $n$, that is, the minimal $r \ge 1$ with $\arb^\wt(r) = n$. In
the time between $t=0$ and $t=t_1$, the particles $1,\dots,r-1$ get absorbed
by $r$ and the particles $r+1,\dots,n-1$ get absorbed by $n$. The
corresponding collision patterns are precisely $\arbA$ and $\arbAA$,
respectively.

The particle perspective also allows us to give an interpretation of the
support function of $\PP(\Simp,c)$.  For given weakly increasing $-\wt$, the
\Def{lifespan} of a particle $i < n$ is the timespan until it is absorbed by
some particle $j > i$.

\begin{prop}\label{prop:av_lifespan}
    For $-\wt$ weakly increasing, $-h_{\PP(\Simp,c)}(\wt)$ is the 
    sum of lifespans of all particles $i=1,\dots,n-1$.
\end{prop}
\begin{proof}
    We compute
    \[
        h_{\PP(\Simp,c)}(\wt)  =   \max_{\arb} \inner{\wt,\PPGKZ(\arb)}
 =  - \min_{\arb} \inner{-\wt,\PPGKZ(\arb)}
         =  -\min_{\arb} \sum_{i < n} \frac{-(\wt_{\arb(i)} -
        \wt_i)}{c_{\arb(i)} -
        c_i}
         =  - \sum_{i < n} \min_{j>i} \frac{-(\wt_{j} - \wt_i)}{c_{j} -
        c_i}
    \]
    and thus $-h_{\PP(\Simp,c)}(\wt)$ is the sum of lifespans of all particles
    $i < n$.
\end{proof}

\newcommand\br{\mathcal{B}}%
Collisions of particles can be encoded in terms of \emph{bracketings}.  A
\Def{bracket} is a subset of $[n]$ of the form $[i,j] := \{i, i+1,\dots,j\}$
for $i \le j$. A \Def{bracketing} $\br$ is a collection of distinct brackets
$B_1,\dots,B_m$ such that for every $1 \le r < s \le m$, $B_r \subset B_s$ or
$B_s \subset B_r$ or $B_s \cap B_r = \emptyset$.  A bracket represents
particles that have collided with each other at some point in time.  If $B_r
\subset B_s$ are contained in a bracketing $\br$, then this means at some
point all the particles in $B_r$ have collided and later all the particles in
$B_s$ will have collided. As there is ultimately only a single particle
left, every bracketing has to contain $[1,n]$. When convenient, we will also
assume that $\br$ contains all singleton brackets $[i,i]$. The
\Def{associahedron} $\Ass[n-2]$ is the set of bracketings of $[n]$ ordered by
\emph{reverse} inclusion. The unique \emph{maximal} element is $\{[1,n]\}$.
The minimal elements are the complete bracketings of $[n]$, that is,
bracketings with $n-1$ brackets. \Cref{fig:ass} gives an example.

\begin{figure}[h]
    \begin{center}
        \includegraphics[height=3.5cm]{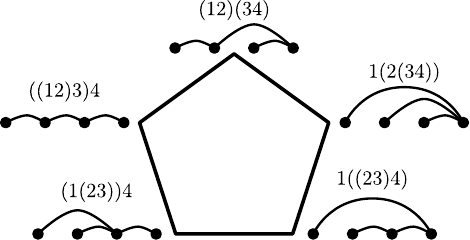}
    \end{center}
    \caption{The $\PP(\Simp[3],c)$ is combinatorially
    isomorphic to $\Ass[2]$, which is a pentagon. The figure shows the
    corresponding labelling.}
    \label{fig:ass}
\end{figure}

\begin{thm}\label{thm:main_asso}
    Let $P$ be an $(n-1)$-dimensional simplex and $c$ a generic objective
    function. Then $\PP(P,c)$ is combinatorially isomorphic to the
    $(n-2)$-dimensional associahedron $\Ass$.
\end{thm}

As a first step towards proving \Cref{thm:main_asso}, we translate collision
patterns into complete bracketings. \Cref{fig:bracket} illustrates the
correspondence.

\begin{figure}[h]
    \begin{center}
    \includegraphics[height=2cm]{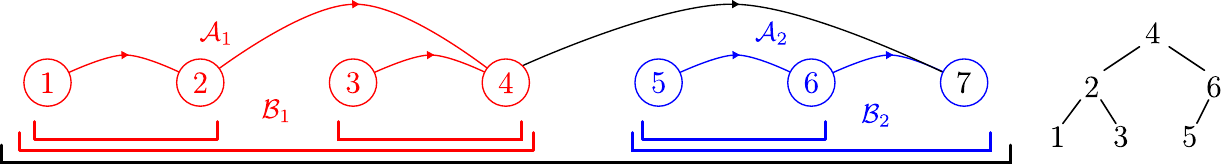}
    \end{center}
    \caption{Left: the arborescence of \Cref{fig:noncross} together with the
    associated bracketing. Right: the associated partial order.}
    \label{fig:bracket}
\end{figure}

\begin{thm}\label{thm:arb_to_cs}
    Collision patterns on $n$ particles are in bijection with complete bracketings.
\end{thm}
\begin{proof}
    For $n \le 2$, there is only one collision pattern and one complete
    bracketing. Let $\arb$ be a collision pattern on $n > 2$ particles. Let
    $\arbA : [r-1] \to [r]$ and $\arbAA : [n-r-1] \to [n-r]$ be the
    decomposition of \Cref{lem:decomp}. By induction, $\arbA, \arbAA$ give
    rise to bracketings $\br_1, \br_2$ on $r$ and $n-r$ particles,
    respectively. After the particles in $\br_2$ are relabelled from $[n-r]$
    to $\{r+1,\dots,n\}$, we get a complete bracketing $\br = \br_1 \cup \br_2
    \cup \{[1,n]\}$.

    Conversely, if $\br$ is a complete bracketing, then $\br \setminus \{[1,n]
    \}$ decomposes into two complete bracketings $\br_1, \br_2$ of $[1,r]$ and
    $[r+1,n]$. By induction they give rise to noncrossing arborescences
    $\arbA, \arbAA$ and \Cref{lem:decomp} completes the proof.
\end{proof}

For the next result it will be convenient to extend the definition of
collision patterns by setting $\arb(0) := n$.

\begin{lem}\label{lem:bracket}
    Let $\arb$ be a collision pattern with associated bracketing $\br$. The
    bracket $[a,b]$ is contained in $\br$ if and only if $\arb(a-1) \ge b$ and
    $\arb^k(a) = b$ for some $k \ge 1$.
\end{lem}
\begin{proof}
    Assume that $[a,b] \in \br$ and let $r$ be minimal with $\arb(r) = n$.
    Unless $[a,b] = [1,n]$, $[a,b] \in \br_1$ or $[a,b] \in \br_2$, where $\br
    = \br_1 \cup \br_2 \cup \{[1,n]\}$ is the decomposition of the proof of
    \Cref{thm:arb_to_cs}. By induction, we can assume that $[a,b] = [1,r]$ or
    $[a,b] = [r+1,n]$. In the former, $r$ is the sink of $\arb_1$ and hence
    $\arb^{k}(1) = r$ for some $k$ and $\arb(0) = n \ge r$. In the latter
    case, $n$ is the unique sink of $\arb$ and $\arb((r+1)-1) = n$.

    Conversely, let $a<b$ with $\arb(a-1) \ge b$ and $\arb^k(a) = b$ for some
    $k \ge 1$. If $[a,b] = [1,n]$, then $[a,b] \in \br$. Otherwise, let $r \ge
    1$ be minimal with $\arb(r) = n$.  Since $\arb(a-1) \ge b$, noncrossing
    implies $1 \le r \le a$ or $r \ge \arb(a-1) \ge b$. If $r = a$, then this
    implies $b = n$. If $1 \le a-1$ or $r \ge \arb(a-1) \ge b$, then we
    can replace $\arb$ with $\arbA$ or $\arbAA$, respectively and the claim
    follows by induction.
\end{proof}

We note that every weakly decreasing set of locations $w \in \R^n$ yields a
partial bracketing.  For example, as we saw earlier, $-\wt \in \R_{\ge0} c + \R
\1$ yields the coarsest bracketing $\br = \{ [1,n] \}$. This point of view
allows us to explicitly construct locations for single brackets.

\begin{prop}\label{prop:rs_loc}
    Let $1 \le r < s \le n$ with $(r,s) \neq (1,n)$.  The bracketing $\br = \{
        [r,s], [1,n]\}$ is obtained from locations $-\wt$ if and only if 
    \[
        -\mu \wt + \gamma \1 - \alpha c \ = \ -\wt^{r,s} \ := \
        (c_1,\dots,c_{r-1},c_s,\dots,c_s,c_{s+1},\dots,c_n) 
    \]
    for some $\gamma, \alpha, \mu \in \R$ with $\alpha \ge 0$ and $\mu > 0$.
    Moreover, $-h_{\PP(\Simp,c)}(\wt^{r,s}) = n-1-(s-r)$.
\end{prop}
\begin{proof}
    If $-\wt$ yields the bracketing $\{[r,s],[1,n]\}$, then there are exactly
    two points in time $t_1 < t_2$ when collisions happen. At time $t_1$, the
    particles $[r,s]$ simultaneously collide and at $t_2$, all the remaining
    particles $1,\dots,r-1,s,s+1,\dots,n$ simultaneously collide.  Using
    \Cref{cor:PP_inv}, we can replace $-\wt$ by $-\wt' = \frac{1}{t_2-t_1}(-
    \wt - t_1 c)$ and assume that $t_1 = 0$ and $t_2 = 1$.  
    Hence $-\wt'$ satisfies $\wt'_r = \cdots = \wt'_s$ and $\wt'_i -
    c_i = \gamma$ for some $\gamma$ and all $i < r$ and $i \ge s$. As for the
    sum of lifespans with respect to locations $-\wt^{r,s}$, we note that
    particles $r,\dots,s-1$ have lifespan $0$ and particles $i < r$ and $s \le
    i < n$ have lifespan $1$. Thus $-h_{\PP(\Simp,c)}(\wt^{r,s}) = n-1-(s-r)$
    by \Cref{prop:av_lifespan}.
\end{proof}

Note that the faces described in \Cref{prop:rs_loc} have a unique normal
relative to the affine hull of $\PP(\Simp,c)$ and hence are facets.
We write $[a,b]^c$ for the complement $[n] \setminus [a,b]$.

\begin{prop}\label{prop:vertices_in_face}
    Let $\arb$ be a collision pattern and $1 \le r < s \le n$ with $(r,s) \neq
    (1,n)$. Then $\PPGKZ(\arb)$ is contained in the face $\PP(\Simp,c) \cap \{
        x : \inner{-\wt^{r,s},x} = n-1-(s-r) \}$ if and only if $\arb$ maps
    $[r,s-1]$ into $[r,s]$ and $[r,s-1]^c$ into $[r,s-1]^c$.
\end{prop}
\begin{proof}
    Since $-\wt^{r,s}$ is weakly increasing, we get from \Cref{prop:rs_loc}
    that $\inner{-\wt^{r,s},\PPGKZ(\arb)} \ge n-1-(s-r)$. Equality holds if
    and only if $\frac{-(\wt^{r,s}_{\arb(i)} - \wt^{r,s}_i)}{c_{\arb(i)} -
    c_i}$ is minimal for all $i < n$. For $i < r$, this is the case if
    $\arb(i) < r$ or $\arb(i) \ge s$, in which case the value is $1$. For $r
    \le i < s$, the minimal value $0$ is attained if and only if $\arb(i) \le
    s$. 
\end{proof}

Bracketings that contain of a single bracket $[r.s] \neq [1,n]$ correspond to
the facets of any polytopal realization of $\Ass$. \Cref{prop:rs_loc} shows that
$\PP(\Simp,c)$ has a unique facet for every such bracketing.
\Cref{prop:vertices_in_face} gives a combinatorial characterization of which
vertices are contained in those facets and is a decisive step in showing that
$\Ass$ and $\PP(\Simp,c)$ have the same vertex-facet incidences. However
at this point, we cannot exclude that $\PP(\Simp,c)$ has more facets. For that,
we need the following general result.

\begin{lem}\label{lem:comb_iso}
    Let $P$ and $Q$ be two polytopes of the same dimension. Assume that $f
    : V(P) \to V(Q)$ is a bijection on vertices such that for every facet $F
    \subset P$, there is a face $G \subset Q$ with $f(V(F)) = V(G)$. Then
    $P$ and $Q$ are combinatorially isomorphic.
\end{lem}
\begin{proof}
    Every face is uniquely identified by its set of vertices and we can view
    the face lattice $L(P) = \{ V(F) : F \subseteq P \text{ face}\}$ as a
    subposet of $2^{V(P)}$. Since every face is an intersection of facets, $f$
    extends to an injective map $f : L(P) \to L(Q)$. Therefore, we can view
    $L(P)$ as a graded sublattice of $L(Q)$. Since both polytopes have the
    same dimension, it follows that $f(F)$ is a facet of $Q$ for all facets $F
    \subset P$. Assume that there is a facet of $Q$ not contained in $L(P)$.
    The dual graph of $Q$ is connected and thus there exists a facet $G \in
    L(Q) \setminus L(P)$ and a facet $F \in L(P)$ such that $\dim F \cap G =
    \dim P - 2 = \dim Q - 2$. However, since $L(P)$ is the face lattice of a
    polytope and $F$ is a facet, there is a unique facet $F' \in L(P)$ with $F
    \cap G = F \cap F'$. Since $F'$ is also a facet of $Q$, this implies $G =
    F'$.
\end{proof}

\begin{proof}[Proof of \Cref{thm:main_asso}]
    Using \Cref{cor:PP_inv} we can assume that $P = \Simp$ and $0 < c_1 <
    \cdots < c_n$.

    Let $\arb$ be a collision pattern with bracketing $\br$ and  $1 \le r < s \le
    n$ with $[r,s] \neq [1,n]$.  We first prove that $\PPGKZ(\arb)$ is
    contained in the face $\PP(\Simp,c)^{\wt^{r,s}}$ if and only if $[r,s] \in
    \br$.  From \Cref{prop:vertices_in_face} we get that $\PPGKZ(\arb) \in
    \PP(\Simp,c)^{\wt^{r,s}}$ if and only if $\arb([r,s-1])\subseteq [r,s]$
    and $\arb([r,s-1]^c)\subseteq [r,s]^c$.  If $\arb([r,s-1])\subseteq
    [r,s]$, then there is some $k$ such that $\arb^k(r) = s$. Moreover, if
    $\arb([r,s-1]^c)\subseteq [r,s]^c$, then $\arb(r-1) \ge s$.
    \Cref{lem:bracket} now implies that $[r,s] \in \br$. Conversely, if $\arb$
    satisfies $\arb(r-1) \ge s$ then $\arb(i) \in [1,r-1] \cup [s,n]$ for all
    $1 \le i \le r-1$. Since trivially $\arb([s,n]) \subseteq [s+1,n]$, this
    shows $\arb([r,s-1]^c)\subseteq [r,s]^c$. If there is $i \in [r,s-1]$ with
    $\arb(i) > s$, then there is no $k$ such that $\arb^k(r) = s$. Hence,
    $\arb([r,s-1]) \subseteq [r,s]$, which proves the claim.

    By \Cref{thm:arb_to_cs}, we have a bijection $f$ from complete bracketings
    to collision patterns, that is, vertices of $\PP(\Simp,c)$. The facets of
    $\Ass$, that is, the maximal elements of $\Ass$ correspond to the
    bracketings $\{[r,s],[1,n]\}$ with $1 \le r < s \le n$ and $(r,s) \neq
    (1,n)$. The claim above now states $[r,s] \in \br$ if and only if
    $\PPGKZ(f(\br)) \in \PP(\Simp,c)^{\wt^{r,s}}$. Finally, $\Ass$ is the face
    lattice of a polytope of dimension $d-2 = \dim \Pi(\Simp,c)$ and
    \Cref{lem:comb_iso} completes the proof.
\end{proof}

\begin{remark}
    A simpler proof of \Cref{thm:main_asso} proceeds as follows. From the
    proof of \Cref{thm:noncross} is it not difficult to determine when
    $[\PPGKZ(\arb),\PPGKZ(\arb')]$ is an edge of $\PP(\Simp,c)$. In
    particular, one sees that the graph of $\PP(\Simp,c)$ is $(d-2)$-regular
    and hence $\PP(\Simp,c)$ is a simple polytope.  Blind and
    Mani~\cite{BlindMani} showed that two simple polytopes are combinatorially
    isomorphic if and only if they have isomorphic graphs; see
    also~\cite[Sect.~3.4]{zieglerbook}. \Cref{thm:arb_to_cs} actually gives an
    isomorphism of graphs and completes the argument. However, the pivot rule
    polytopes of products of simplices are not simple and the correspondence
    to constrainahedra will need a similar argument.
\end{remark}

In the next section, we need yet another representation of particle
collisions. For a collision pattern $\arb$, we define a partial order by
setting $i \prec j$ if particle $i$ has to be absorbed before particle $j$ can
be absorbed. Since the particle $n$ is never absorbed, this is a partial order
on $[n-1]$. For example, in the collision pattern of Figure~\ref{fig:bracket},
we note that $3$ needs to be absorbed by $4$ before $2$ can be absorbed.
However, since $2$ absorbs $1$, this has to happen before $2$ can absorbed by
$4$. Hence $1 \prec 2$ and $3 \prec 2$. There is no dependence between $1$ and
$3$ and they are not comparable with respect to $\preceq$. The complete
partial order is shown in \Cref{fig:bracket}. 

It is straightforward to see that $i$ is a minimal element with respect to
$\preceq$ if and only if $i$ is an immediate leaf.  Moreover, there is a unique element $j$ that covers $i$, that is, $i
\prec j$ and there is no $j'$ with $i \prec j' \prec j$.  Let $j$ be the
maximal $k < i$ with $\arb(k) = \arb(i)$. If such a $j$ exists, the noncrossing condition demands $j
= i-1$. By removing $i$ from the collision pattern (and relabelling), we see
that $j$ becomes an immediate leaf and that $i$ was the only obstruction.
If there is no $k < i$ with $\arb(k) = \arb(i) = i+1$, then $j = \arb(i) =
i+1$. Indeed, once all particles $j+1,\dots,\arb(j)-1$ have been absorbed by
$\arb(j)$, $i$ is the only obstruction for $j$ becoming an immediate leaf.
Observe that in the Hasse diagram of $\preceq$ every non-maximal element is
covered by a unique element and every element covers at most two elements. The
Hasse diagram is a binary search tree.

\begin{lem}\label{lem:poset}
    Let $\arb$ be a collision pattern on $n \ge 2$ particles with associated
    partial order $\preceq$. For $a, b \in [n-1]$, we have $a \preceq b$ if
    and only if $a \le b$ and $b = \arb^k(a)$, or 
        $b < a$ and $\arb(b) = \arb^{k+1}(a)$ for some $k \ge 0$.
\end{lem}
\begin{proof}
    We prove the equivalence by induction on $n$. If $n = 2$, then $a=b$ and
    the equivalence is trivially true. Assume that $n > 2$ and $a \neq b$. By
    \Cref{lem:immediate}, $\arb$ has an immediate leaf.  Assume that there is
    an immediate leaf $c \not\in \{a,b\}$. As discussed before,
    $c$ is a minimum of $\preceq$ that can be removed without changing the
    order relation between $a$ and $b$. Similarly, $c$ can be removed from the
    collision pattern without interfering with the stated conditions. In this
    case, the equivalence holds by induction.

    If no such $c$ exists, then $a$ or $b$ has to be the unique immediate leaf
    of $\arb$. If $b$ is the immediate leaf then neither condition can hold.
    Thus $a$ is an immediate leaf and there is a unique $j$ that covers $a$.
    Removing $a$, we get by induction that $j \preceq b$ if and only if $j \le
    b$ and there is $k \ge 0$ with $b = \arb^k(j)$ or $j > b$ and $\arb(b) =
    \arb^k(j)$ for some $k > 0$.  By the preceding discussion, we have that $j
    = \arb(a)$ or $\arb(j) = \arb(a)$. In the former case, we have $b =
    \arb^k(j)$ if and only if $b = \arb^{k+1}(a)$ and $\arb(b) = \arb^k(j)$ if
    and only if $\arb(b) = \arb^{k+1}(a)$. In the latter case, we have $b =
    \arb^k(j)$ if and only if $b = \arb^k(a)$ and $\arb(b) = \arb^k(j)$ if and
    only if $\arb(b) = \arb^k(a)$.
\end{proof}

\section{Products of simplices and constrainahedra}
\label{sec:products}

In this section we investigate the max-slope pivot rule polytopes of the
product of two simplices. As in the case of simplices, the construction of
max-slope arborescences lends itself to an interpretation in terms of
particle collisions. 

Let $m,n \ge 1$ and consider the Cartesian product of an $(m-1)$-simplex and
an $(n-1)$-simplex. Appealing to \Cref{cor:PP_inv}, it suffices to
consider\footnote{Note that in the simplex case it would have sufficed to
require that $P$ is combinatorially isomorphic to a simplex. Here we have to
assume that the polytope is affinely isomorphic to a Cartesian product of
simplices.}
\[
    \PS \ := \ \Simp[m-1] \times \Simp[n-1] \ \subset \ \R^{m} \times \R^{n} \
    \cong \ \R^{m + n} \, .
\]
Let us write $e'_r = e_r$ for $r=1,\dots,m$ and $e''_i = e_{m+i}$ for
$i=1,\dots,n$. Then $\PS$ is a simple polytope with $m\cdot n$ vertices
$e'_r+e''_i$ for $r=1,\dots,m$ and $i = 1,\dots,n$, contained in the
codimension-$2$ subspace of all $(x',x'') \in \R^m \times \R^n$ with $\sum_r
x'_r = \sum_i x''_i = 1$.  We choose an edge-generic objective function $c =
(c',c'') \in \R^m \times \R^n$, for which we can assume that $0 < c'_1 < c'_2
< \cdots < c'_m$ and $0 < c''_1 < c''_2 < \cdots < c''_n$.  The graph of $\PS$
oriented by $c$ has nodes $\PSV := [m] \times [n]$ that we write as $ri$ for
$r \in [m]$ and $i \in [n]$ and edges of the form $ri \to rj$ for $i < j$
(vertical edges) and $ri \to si$ for $r < s$ (horizontal edges). The unique
sink is the node $mn$ and we write $\PSVo := \PSV \setminus \{ mn \}$.

Arborescences of $(\PS,c)$ can be identified with maps $\arb : \PSVo \to \PSV$
with the property that if $\arb(ri) = sj$ then $r < s$ and $i=j$ or $r = s$
and $i < j$. The number of arborescences is $\prod_{i=1}^{m-1}
\prod_{j=1}^{n-1} (i+j)$. Every arborescence $\arb$ determines an arborescence
$\arb'$ of $(\Simp[m-1],c')$ by $\arb'(r) := s$ if $\arb(rn) = sn$.
Analogously, we obtain an arborescence $\arb''$ for $(\Simp[n-1],c'')$.

Let $\wt = (\wt',\wt'') \in \R^m \times \R^n$ be generic. In order to determine
the max-slope arborescence of $\PS$ with respect to $\wt$, let $\arb^{\wt'} :
[m-1] \to [m]$ be the max-slope arborescence of $(\Simp[m-1],c')$ determined by
$\wt'$.  If $\arb^{\wt'}(r) = s$, then $s$ maximizes the slope $\tau'(r,s) =
\frac{\wt'_s - \wt'_r}{c'_s - c'_r}$ and we record the maximal slope at $r$ as
$\tau'(r) := \tau'(r,s)$. Likewise, $\arb^{\wt''}$ is the max-slope determined
by $\wt''$ and we define $\tau''$ analogously.  Evaluating
\eqref{eqn:max_slope}, we see that the max-slope $\arb^\wt$ on $(\PS,c)$ is
given by
\begin{equation}\label{eqn:prod_arb}
    \arb^\wt(ri) \ = \
    \begin{cases}
        \arb^{\wt'}(r)i & \text{if } \tau'(r) > \tau''(i) \\
        r\arb^{\wt''}(i) & \text{if } \tau'(r) < \tau''(i). \\
    \end{cases}
\end{equation}
\Cref{fig:grid_noncross} shows an example.
\begin{figure}[h]
    \begin{center}
        \includegraphics[height=5cm]{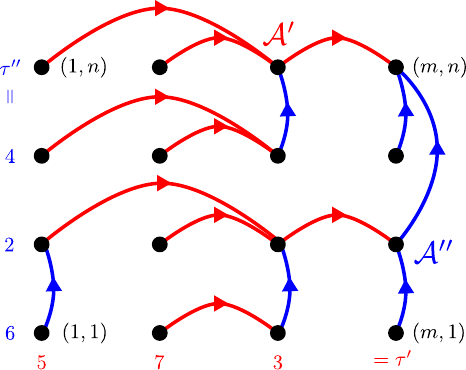}
    \end{center}
    \caption{Example of a max-slope arborescence $\arb^\wt$ for $(\PS,c)$ with
        $(m,n) = (4,4)$, $c' = (1,2,3,4)$, and $c'' = (5,6,7,8)$. The weights
        are $\wt' = (11, 14, 21, 24)$ and $\wt'' = (38, 44, 44, 48)$. The
        slopes $\tau'(r)$ and $\tau''(i)$ are indicated below.}
    \label{fig:grid_noncross}
\end{figure}
In order to describe them combinatorially, we make the following definitions.
A node $ri \in \PSV$ is an \Def{immediate leaf} of $\arb$ if it has no incoming
edges and $\arb(ri) = (r+1)i$ or $\arb(ri)=r(i+1)$. 

\begin{defn}
    For $m,n \ge 1$ an arborescence $\arb : \PSVo \to \PSV$ is \Def{reducible}
    if $m=n=1$ or
    \begin{itemize}[$\circ$]
        \item there exists $r \in [m]$ such that $ri$ is a immediate leaf
            for every $i$ and the restriction of $\arb$ to $\PSVo
            \setminus (r \times [n])$ is reducible, or
        \item there exists $i \in [n]$ such that $ri$ is a immediate leaf
            for every $r$ and the restriction of $\arb$ to $\PSVo
            \setminus ([m] \times i)$ is reducible.
    \end{itemize}
\end{defn}

\Cref{fig:grid_noncross_reduce} shows that the arborescence in
\Cref{fig:grid_noncross} is reducible.  The existence of rows or columns of
immediate leaves is the key property for many of the arguments in this section.
Note that Lemma~\ref{lem:immediate} shows that noncrossing arborescences are
reducible. In fact, the next observation shows that being \emph{noncrossing} is
a consequence of reducibility. Let $\arb : \PSVo \to \PSV$ be a reducible
arborescence and recall that $\arb$ induces arborescences $\arb'$ and $\arb''$
on $(\Simp[m-1],c')$ and $(\Simp[n-1],c'')$, respectively.  We call $\arb$
\Def{consistent} if $\arb(ri) = si$ implies $\arb'(r)=s$ and $\arb(ri) = rj$
implies $\arb''(i)=j$ for all $ri \in \PSVo$.  We call $\arb$
\Def{grid-noncrossing} if $\arb'$ and $\arb''$ are noncrossing and there are no
$1 \le r \le s < t \le m$ and $1 \le i \le j < k \le n$ with $(r,i) \neq (s,j)$
and with $\arb(si) = sk$ and $\arb(rj) = tj$.  We can illustrate an arborescence
$\arb : \PSVo \to \PSV$ as in \Cref{fig:grid_noncross} by drawing arcs
horizontally or vertically between consecutive rows or columns of nodes.
Consistency now means that every arc is a copy of the top-most horizontal arc in
the same column or right-most vertical arc in the same row. Grid-noncrossing
states that every row and column is a partial noncrossing arborescence and there
are is no crossing of a vertical and a horizontal arc.  

\begin{prop}\label{prop:consist_grid-noncross}
    If $\arb$ is reducible, then $\arb$ is consistent and grid-noncrossing.
\end{prop}
\begin{proof}
    We prove the claim by induction on $m+n$. If $m=n=1$, this is trivial. For
    $m+n > 0$, we can remove a row or column of immediate leafs. Both claims
    hold by induction and bringing back the row or column keeps both
    properties intact.
\end{proof}

For $m=1$ or $n=1$, reducible arborescences are precisely the noncrossing
arborescences.  \Cref{fig:non-reduce} shows \Cref{prop:consist_grid-noncross}
does not give an equivalence. It would be interesting to have a non-recursive
characterization of reducible arborescences. The combinatorial notion of
reducibility suffices to completely characterize max-slope arborescences of
products of two simplices.

\begin{figure}[h]
    \begin{center}
        \includegraphics[height=5cm]{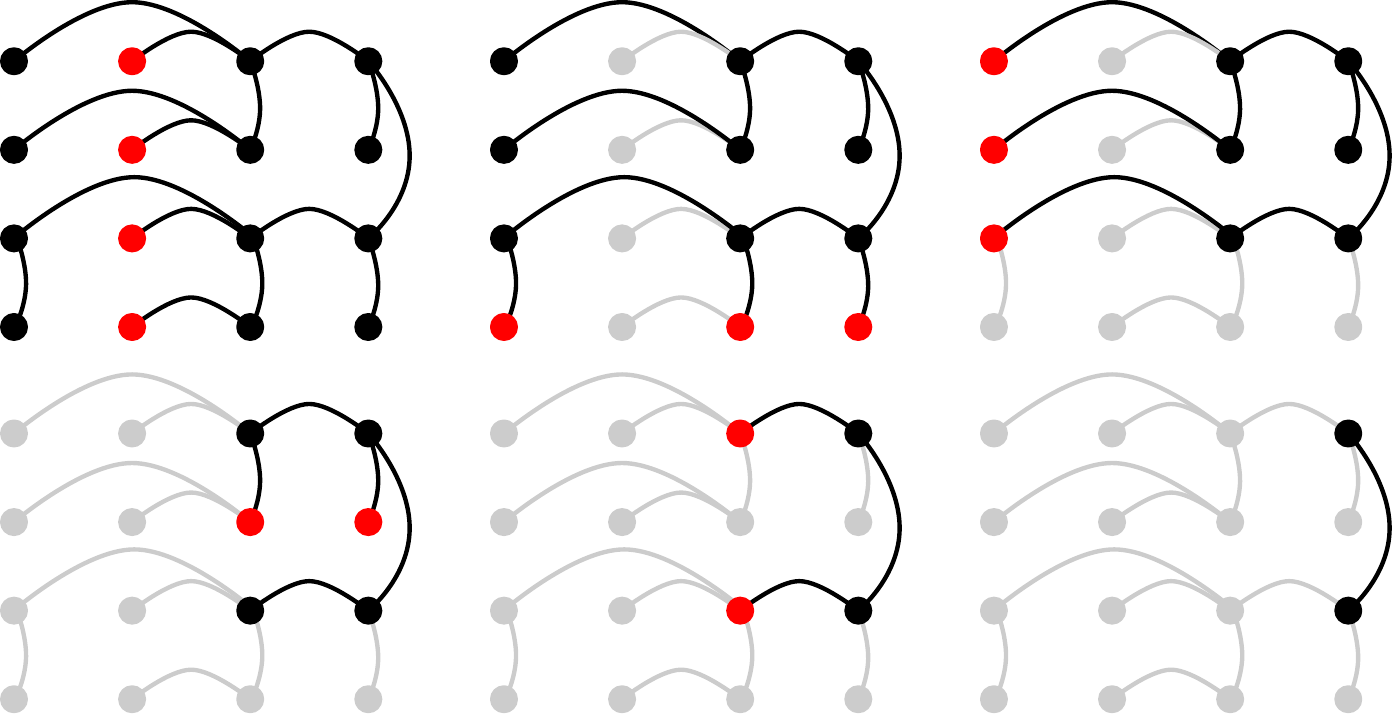}
    \end{center}
    \caption{The figure shows that the arborescence of
    \Cref{fig:grid_noncross} is reducible. The red nodes correspond to rows or
    columns of immediate leaves.}
    \label{fig:grid_noncross_reduce}
\end{figure}

\begin{figure}[h]
    \begin{center}
        \includegraphics[height=1.5cm]{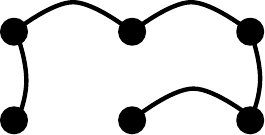}
    \end{center}
    \caption{An example of a consistent and grid-noncrossing arborescence that
    is not reducible.}
    \label{fig:non-reduce}
\end{figure}

\begin{thm}
    An arborescence $\arb : \PSVo \to \PSV$ is a max-slope arborescence of
    $(\PS,c)$ if and only if $\arb$ is reducible.
\end{thm}

\begin{proof}
    Let $\arb = \arb^\wt$ be a max-slope arborescence for some $\wt =
    (\wt',\wt'')$.  We show by induction on $m+n$ that $\arb^\wt$ is reducible.
    The base case $m=n=1$ is trivially true. Let $r$ with $\tau'(r)$ maximal and
    $i$ with $\tau''(i)$ maximal. Without loss of generality, let us assume that
    $\tau'(r) > \tau''(i)$. From \Cref{cor:imm_leaf} we know that $r$ is an
    immediate leaf of $\arb'$ and there are no horizontal arcs into $rj$ for all
    $j$. Moreover, \eqref{eqn:prod_arb} shows that $\arb(rj) = (r+1)j$ for all
    $j$. This shows that $rj$ is an immediate leaf for all $j$. We also note
    that the restriction of $\arb^{\wt}$ to $\PSVo \setminus r \times [n]$ is
    the max-slope arborescence of $\PS[m-2,n-1]$ for $\hat{c} =
    (c'_1,\dots,c'_{r-1},c'_{r+1},\dots,c'_m, c'')$ and $\wt = (\wt'_1, \dots,
    \wt'_{r-1}, \wt'_{r+1}, \dots, \wt'_m,\wt'')$. The claim follows by
    induction.

    To prove the converse, let $\arb : \PSVo \to \PSV$ be a reducible
    arborescence. We prove the existence of $\wt = (\wt',\wt'') \in \R^{m+n}$
    by induction on $(m,n)$. If $m = 1$ or $n=1$, then $\arb$ is a noncrossing
    arborescence and the claim reduces to \Cref{thm:noncross}.  Thus, let $m >
    1$ and $n > 1$. By reducibility and without loss of generality, we may
    assume that $ri$ is an immediate leaf of $\arb$ for all $i$. The
    restriction of $\arb$ to $r \times [n]$ is a reducible arborescence for
    $(\PS[m-2,n-1],\hat{c})$ with $\hat{c}$ as above and by induction there is
    a suitable $\hat{\wt} =
    (\wt'_1,\dots,\wt'_{r-1},\wt'_{r+1},\dots,\wt'_m,\wt''_1,\dots,\wt''_n)$
    that proves that the restriction is a max-slope arborescence.  In order to
    find $\wt'_r$, we obtain from the proof of \Cref{thm:noncross}, that the
    inequalities $L_{rs}(\wt') > 0$ with $r < s$ pose only upper bounds on
    $\wt'_r$. In particular, we can choose $\wt'_r$ such that $\tau'(r) >
    \tau'(s)$ for $s \neq r$. Similarly, let $i \in [n]$ with $\tau''(i)$
    maximal. The condition $\tau'(r) > \tau''(i)$ is again an upper bound on
    $\wt'_r$. Thus for $\wt'_r$ sufficiently small, $\arb = \arb^\wt$ for $\wt
    = (\wt',\wt'')$.
\end{proof}

To view max-slope arborescences of $\PS$ as particle collisions, we use the
following generalization of particles on a line due by Bottman and
Poliakova~\cite{BottmanPoliakova}. We consider $m$ vertical lines
$L'_1,\dots,L'_m$ and $n$ horizontal lines $L''_1,\dots,L''_n$. The lines are
labelled left-to-right and bottom-to-top.  We place a particle $ri$ at the point
of intersection of $L'_r$ and $L''_i$.  The particle movements are induced by
parallel displacements of the $m+n$ lines. In this scenario parallel lines
collide and particles contained in the colliding lines merge.

We equip every line $L'_r$ and $L''_i$ with a location $-\wt'_r$ and $-\wt''_i$
at time $t=0$ and we assume that $-\wt'_1 \le -\wt'_2 \le \cdots \le -\wt'_m$
and $-\wt''_1 \le -\wt''_2 \le \cdots \le -\wt''_n$.  For $t > 0$, the lines
$L'_r$ move to the left with constant velocity $-c'_r$, the lines $L''_i$ move
down with constant velocity $-c''_i$. As before we assume that $0 < c'_1 < c'_2
< \cdots < c'_m$ and $0 < c''_1 < c''_2 < \cdots < c''_n$. If two or more lines
collide, they are absorbed by the line with the largest index and this line
continues at its original speed.  If we assume that the locations $-\wt =
(-\wt',-\wt')$ are sufficiently generic, then no more than two lines collide and
we can encode the collisions of particles by a collision pattern $\colpat^{-\wt}
: \PSVo \to \PSV$; see \Cref{fig:particle_grid} for an illustration.

\begin{figure}[h]
    \begin{center}
        \includegraphics[height=5cm]{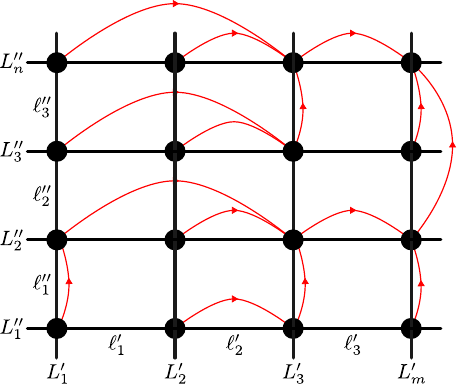}
    \end{center}
    \caption{The figure shows the $m=4$ vertical lines $L'_r$ and the $n=4$
    horizontal lines $L''_i$ together with the $16$ particles at the
    intersections and the collision pattern from \Cref{fig:grid_noncross}.}
    \label{fig:particle_grid}
\end{figure}

The same argument as for \Cref{thm:cs_max_arb} yields the following.
\begin{thm}\label{thm:cs_red_arb}
    Let $m,n \ge 2$ and $c = (c',c'') \in \R^{m+n}$ with $0 < c'_1 < c'_2 <
    \cdots < c'_m$ and $0 < c''_1 < c''_2 < \cdots < c''_n$.  For $\arb :
    \PSVo \to \PSV$ and $\wt = (\wt',\wt'') \in \R^{m+n}$, the following are
    equivalent:
    \begin{enumerate}[\rm i)]
        \item $\arb = \colpat^{-(\wt - \alpha c)}$ is a collision pattern for
            $m\cdot n$ particles contained in $m$ vertical lines and $n$
            horizontal lines moving at velocities $c=(-c',-c'')$ and starting
            at locations $-(\wt - \alpha c)$ for some $\alpha \ge 0$.
        \item $\arb = \arb^\wt$ is a max-slope arborescence of
            $(\PS,c)$ with respect to weight $\wt$.
    \end{enumerate}
\end{thm}

The combinatorics of collisions of $m\cdot n$ particles sitting on $m$
vertical and $n$ horizontal lines was modelled in~\cite{BottmanPoliakova} by
certain preorders on $\L = \{L'_1, \dots, L'_{m-1}, L''_1, \dots,
L''_{n-1}\}$. Recall that a \Def{preorder} on $\L$ is a reflexive and
transitive relation $\preceq$. For $x \in \L$, the equivalence class $[x]$ is
the collection of elements $y \in \L$ with $x \preceq y$ and $y \preceq x$. On
the collection of equivalence classes $\preceq$ yields a partial order.
Bottman and Poliakova define a preorder $\preceq$ on $\L$ to be a \Def{good
rectangular} preorder~\cite[Def.~2.1]{BottmanPoliakova} if 
\begin{itemize}[$\circ$]
    \item $L'_r$ and $L''_i$ are comparable for all $r,i$ 
        \hfill(orthogonal comparability)
\end{itemize}
and $L'_o \preceq L'_q$ (respectively $L''_o \preceq L''_q$) if and only if
\begin{itemize}[$\circ$]
    \item $L'_o \preceq L''_p \preceq L'_q$ (respectively $L''_o \preceq L'_p
        \preceq L''_q$) for some $p$, or \hfill (orthogonal link)
    \item there is no $\min(o,q) < p < \max(o,q)$ with $L'_o \prec L'_p \succ
        L'_q$ (respectively $L''_o \prec L''_p \succ L''_q$).  \hfill (no
        gaps)
\end{itemize}

From a collision pattern $\arb$, we construct a partial order $\preceq$ on $\L$
that captures which lines have to be absorbed before other lines can be
absorbed. The lines $L'_m$ and $L''_n$ are never absorbed. Let $\preceq'$ and
$\preceq''$ be the partial orders on $L'_1,\dots,L'_{m-1}$ and
$L''_1,\dots,L''_{n-1}$ obtained from $\arb'$ and $\arb''$, respectively;
cf.~\Cref{lem:poset} and the discussion preceding it. We define the binary
relation $\preceq$ by
\begin{enumerate}[(P1)]
    \item $L'_r \preceq L''_i$ if $\arb(ri) = si$ and $L''_i \preceq L'_r$ if $\arb(ri) = rj$;
    \item $L'_r \preceq L'_s$ if $L'_r \preceq' L'_s$ or $L'_r \preceq
        L''_i \preceq L'_s$  for some $i$;
    \item $L''_i \preceq L''_j$ if $L''_i \preceq'' L''_j$ or $L''_i
        \preceq L'_r \preceq L''_j$ for some $r$.
\end{enumerate}

\begin{prop}\label{prop:grid_poset}
    The binary relation $\preceq$ is a good rectangular partial order on $\L$.
\end{prop}
\begin{proof}
    We first prove that $\preceq$ is a partial order.  By reducibility, let us
    assume that $ri$ is an immediate leaf for all $i$.  This implies that
    $L'_r \prec L''_i$ for all $i$. Moreover, $r$ is an immediate leaf of
    $\arb'$ and hence $L'_r$ is a minimum with respect to $\preceq'$. This
    implies that if $x \preceq L'_r$ then $x = L'_r$.  By induction, this
    shows that $\preceq$ restricted to $\L \setminus \{L'_r\}$ is a partial
    order and that $\preceq$ is reflexive and anti-symmetric on $\L$. We are
    left to show that if $L'_r \prec x$ and $x \prec y$, then $L'_r \prec y$.
    If there is some $L''_i$ with $x \preceq L''_i \preceq y$, then we are
    done. Hence $x = L'_s$ and $y = L'_t$ for some $s,t$ and $L'_s \prec'
    L'_t$. Since $\preceq'$ is transitive, we get $L'_r \preceq' L'_t$ and
    hence $L'_r \preceq L'_t$.

    To show that $\preceq$ is good rectangular, we only need to show that the
    \emph{no gaps} condition characterizes $\preceq'$ (and hence $\preceq''$).
    Assume that $1 \le r < s < t \le m-1$. From \Cref{lem:poset} we get that
    $L'_r \prec L'_s$ if and only if $s = (\arb')^k(r)$ for some $k > 0$ and
    $L'_t \prec L'_s$ if and only if $\arb'(s) = (\arb')^l(t)$ for some $l >
    0$. There is no $h$ with $t = (\arb')^h(r)$ and since $\arb(r) \le s$,
    there is no $h$ with $(\arb')^h(t) = \arb(r)$. Thus $L'_r$ and $L'_t$ are
    incomparable. The converse argument is analogous.
\end{proof}

\Cref{prop:grid_poset} gives a way to build up the partial order by
successively removing rows or columns of immediate leaves.
\Cref{fig:grid_poset} gives an example. Note that the Hasse diagram is not a
tree anymore. Removing a row/column of immediate leaves might open up
rows/columns of immediate leaves.

\begin{figure}[h]
    \begin{center}
        \includegraphics[height=4cm]{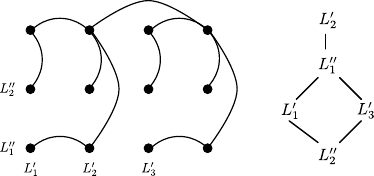}
    \end{center}
    \caption{A collision pattern and the corresponding poset.}
    \label{fig:grid_poset}
\end{figure}

\Cref{prop:grid_poset} shows that the partial order associated to a collision
pattern is a good rectangular poset. In fact, every good rectangular poset arises
that way.

\begin{prop}\label{prop:coll_rec}
    Collision patterns for $m$ vertical and $n$ horizontal lines are in
    bijection with good rectangular posets.
\end{prop}
\begin{proof}
    This follows by reducibility and induction on $m+n$.
\end{proof}

A good rectangular preorder $\preceq_2$ is \Def{refined} by $\preceq_1$ if $x
\preceq_1 y$ implies $x \preceq_2 y$ for all $x,y \in \L$. 

\begin{thm}[{\cite[Theorem~4.1]{BottmanPoliakova}}]\label{thm:def_constrainahedron}
    The good rectangular preorders on $\L = \{L'_1, \dots,
    L'_{m-1}, L''_1, \dots, L''_{n-1}\}$ partially ordered by refinement is
    the face poset of an $(m+n-3)$-dimensional polytope and is called the
    \Def{$(m,n)$-constrainahedron} $C(m,n)$.
\end{thm}

In the rest of this section we will prove the following.

\begin{thm}\label{thm:constrainahedron}
    For $m,n \ge 2$ and a generic $c \in \R^{m+n}$, the max-slope pivot
    polytope $\PP(\PS,c)$ is combinatorially isomorphic to the
    $(m,n)$-constrainahedron.
\end{thm}

Let us call a simultaneous collision of horizontal and vertical lines an
\Def{elementary collision} if it is not the result of at least two
simultaneous and disjoint collisions of sets of lines. These elementary
collisions come in three types for which we give locations for fixed
velocities $c = (c',c'') \in \R^{m+n}$:
\newcommand\wVC{w^\mathrm{VC}}%
\newcommand\wHC{w^\mathrm{HC}}%
\newcommand\wMC{w^\mathrm{MC}}%
\begin{enumerate}
    \item[(VC)] For $1 \le r \le t \le m-1$, the vertical lines $M = \{ L'_s :
        r \le s \le t \}$ are simultaneously absorbed followed by the
        simultaneous collision of all vertical and horizontal lines $\L
        \setminus M$.  Note that the lines in $M$ are absorbed by $L'_{t+1}$.
        The corresponding good rectangular preorder has equivalence classes
        $M$ and $\L \setminus M$.

        We define $-\wVC_{r,t} := (\wt',c'')$ where $\wt'_s := c'_{t+1}$ if $s
        \in [r,t]$ and $\wt'_s := c'_s$ otherwise. The total lifespan is
        $-h_{\PP(\PS,c)}(\wVC_{r,s}) = (n-1)(m-1 - (t-r))$.

    \item[(HC)] For $1 \le i \le k \le n-1$, the horizontal lines $M = \{
            L''_j : i \le j \le k \}$ are simultaneously absorbed (by
        $L''_{k+1}$) followed by the simultaneous collision of all vertical
        and horizontal lines $\L \setminus M$. The corresponding good
        rectangular preorder has equivalence classes $M$ and $\L \setminus M$.

        We define $-\wHC_{i,k} := (c',\wt'')$ where $\wt''_j := c''_{k+1}$ if $j
        \in [i,k]$ and $\wt''_j := c''_j$ otherwise.  The total lifespan is
        $-h_{\PP(\PS,c)}(\wHC_{i,k}) = (m-1)(n-1 - (k-i))$.
\end{enumerate}
\newcommand\I{\mathcal{I}}%
For the last type of elementary collisions, let us call a non-empty collection
of intervals $\I = \{ I_1,\dots,I_g \}$ of $[k]$ \Def{discontinuous} if they
are pairwise disjoint and $I_i \cup I_j$ is not an interval for $i \neq j$. We
write $\bigcup \I = I_1 \cup \cdots \cup I_g$.

\begin{enumerate}
    \item[(MC)] Let $\I'$ and $\I''$ be discontinuous
        collections of $[m-1]$ and $[n-1]$, respectively, such that
        $\I' \neq \{ [m-1] \}$ or $\I'' \neq \{ [n-1] \}$.  The
        collection of lines $M = \{ L'_r, L''_i : r \in \bigcup\I', i
        \in \bigcup\I'' \}$ are simultaneously absorbed followed by the
        simultaneous collision of all lines $\L \setminus M$.

        We define $-\wMC_{\I',\I''} := (\wt',\wt'')$ where $\wt'_s := c'_{t+1}$
        if $s \in [r,t] \in \I'$ and $\wt'_s := c'_{s}$ otherwise and $\wt''_j
        := c''_{k+1}$ if $j \in [i,k] \in \I''$ and $\wt''_j := c'_{j}$
        otherwise. The total lifespan is $-h_{\PP(\PS,c)}(\wHC_{i,k}) =
        (m-1)(n-1) - |\bigcup \I'| \cdot |\bigcup \I''|$.
\end{enumerate}

\newcommand\CRPO{\sqsubseteq}
\begin{proof}[Proof of \Cref{thm:constrainahedron}]
    The proof follows the same strategy as the proof of \Cref{thm:noncross}.
    By \Cref{prop:coll_rec}, there is a bijection between good rectangular
    orders $\preceq$ in $C(m,n)$ and the vertices of $\PP(\PS,c)$, that is,
    reducible arborescences $\arb : \PSVo \to \PSV$. Let $\CRPO$ be a coarsest
    nontrivial good rectangular preorder in $C(m,n)$. We show that there is a
    face $F \subset \PP(\PS,c)$ such that $\preceq$ refines $\CRPO$ if and
    only if $\PPGKZ(\arb) \in F$.  Since $\PP(\PS,c)$ and $C(m,n)$ are
    polytopes of dimension $m+n-3$, \Cref{lem:comb_iso} yields the claim.  For
    $m =1$ or $n=1$, this is precisely the proof of \Cref{thm:noncross}.

    Following~\cite{BottmanPoliakova}, the coarsest good rectangular preorders
    are those associated to elementary collisions of types (VC), (HC), and
    (MC). Let us assume that $\CRPO$ represents an elementary collision of
    type (VC) with $M = \{ L'_s : r \le s \le t \}$ for some $1 \le r \le t
    \le m-1$.  The two other cases are treated analogously.  Let $-\wt =
    -\wVC_{r,s}$ be the locations realizing the elementary collision and let
    $F = \PP(\PS,c)^{-\wt}$. For these locations, we have $\tau''(j) = -1$ for
    all $j \in [n]$ and $\tau'(q)  = 0$ if $q \in [r,t]$ and $\tau'(q)<0$
    otherwise.  In particular, $\PPGKZ(\arb)$ is contained in $F$ if and only
    if $\arb$ has a total lifespan $\inner{-\wt,\PPGKZ(\arb)}$ of
    $(n-1)(m-1-(t-r))$.

    We will use the fact that if $x \in \L$ is minimal with respect to
    $\preceq$, then $\preceq$ refines $\CRPO$ if and only if $x \in M$ and
    $\preceq$ refines $\sqsubseteq$ when restricted to $\L \setminus x$. 

    Assume that $\preceq$ refines $\CRPO$.  If $x$ is minimal with respect to
    $\preceq$, then $x$ corresponds to a row or column of immediate leaves of
    $\arb$. Now $x \in M$ implies $x = L'_s$ for some $s \in [r,t]$ and thus
    $\arb(si) = (s+1)i$ for all $i$. In particular, $\tau'(s) = 0$ and hence
    this row of immediate leaves contributes $0$ to the total live span
    $\inner{-\wt,\PPGKZ(\arb)}$.  That is $\inner{-\wt,e'_{s+1}-e'_{s}} = 0$
    and comparing with \eqref{eqn:GKZ}, we see that deleting the $s$-th
    coordinate of $\PPGKZ(\arb) - (n-1) \frac{e'_{s+1}-e'_s}{c'_{s+1}-c'_s}$
    is precisely $\PPGKZ(\overline{\arb})$, where $\overline{\arb}$ is the
    restriction of $\arb$ to $\PSVo \setminus s \times [n]$.  If we now set
    $\overline{\wt} = (\wt'_1,\dots,\wt'_{s-1},\wt'_{s+1},\dots,\wt'_m,\wt'')
    \in \R^{m-1+n}$, then $\PPGKZ(\arb) \in F = \PP(\PS,c)^{-\wt}$ if and only
    if $\PPGKZ(\overline{\arb}) \in \PP(\PS[m-2,n-1],c)^{-\overline{\wt}}$ and
    the proof follows by induction.

    \newcommand\twt{\widetilde{\wt}}%
    For the converse, let $\PPGKZ(\arb)$ be a vertex of $F =
    \PP(\PS,c)^{-\wt}$.  In particular, there is $\twt \in \R^{m+n}$ such that
    $\arb$ is the max-slope arborescence with respect to $\twt$. For every
    $\eps > 0$, $\arb = \arb^{\wt + \eps \twt}$. For $\eps$ sufficiently
    small, we gather from \eqref{eqn:prod_arb} that a minimum $x$ of $\preceq$
    corresponds to a row of immediate leaves $\arb(si) = (s+1)i$ for $s \in
    [r,t]$ and $i = 1,\dots,n-1$. Hence $x = L'_s \in M$. The same argument as
    above shows that $\PPGKZ(\arb) - (n-1)
    \frac{e'_{s+1}-e'_s}{c'_{s+1}-c'_s}$ descents to a vertex of
    $\PP(\PS[m-2,n-1],c)^{-\overline{\wt}}$ and we are done by induction.
\end{proof}

\section{Higher products and multiplihedra}
\label{sec:higher}

The particle interpretations of the previous section generalize to max-slope
pivot rule polytopes of higher products of simplices
\[
    \PS[n_1,\dots,n_k] \ := \ \Simp[n_1] \times \cdots \times \Simp[n_k] \, , 
\]
by considering collections of hyperplanes parallel to the coordinate planes in
$\R^k$ and with particles sitting at their $0$-dimensional intersections.

If $n_i = 1$ for all $i=1,\dots,k$, then $P$ is linearly isomorphic to a
$k$-dimensional cube and the max-slope pivot rule polytope was determined
in~\cite{PivPoly}. Let $\SymGrp_k$ be the permutations of $[k]$.  Recall that
the \Def{permutahedron} $\Pi(a_1,\dots,a_k)$ for $a_1,\dots,a_k \in \R$ is the
convex hull of $(a_{\sigma(1)},\dots,a_{\sigma(k)})$ for for all $\sigma \in
\SymGrp_k$.  If all $a_i$ are distinct, then $\Pi(a_1,\dots,a_{k})$ is a
simple $(k-1)$-dimensional polytope with $k!$ many vertices;
cf.~\cite[Section~0]{zieglerbook}. 

\begin{prop}[{\cite[Example~4.2]{PivPoly}}]
    If $n_1 = \cdots = n_k = 1$, then 
    the max-slope pivot rule polytope of $\PS[n_1,\dots,n_k]$ is linearly
    isomorphic to the permutahedron $\Pi(2^0, 2^1, \dots, 2^{k-1})$.
\end{prop}

This agrees with the particle perspective, where we would consider for each
$i=1,\dots,k$ two hyperplanes in $\R^k$ parallel to $\{x_i = 0\}$. Collisions
here correspond to the event that two parallel hyperplanes meet and hence the
relevant information is the order in which these events take place.

In this last section, we focus on the max-slope pivot rule polytopes of
$\PS[n-1,1,1,\dots,1] = \Simp[n-1] \times (\Simp[1])^k$. Making use of \Cref{cor:PP_inv}, it suffices
to look at the polytopes
\[
    Q_{n,k} \ := \ \Simp[n-1] \times [0,1]^k \ \subset \ \R^n \times \R^k \, 
\]
together with an objective function $(c,r) \in \R^{n + k}$ with $c = (c_1 <
c_2 < \cdots < c_n)$ and $r = (r_1,\dots,r_k) \in \R^{k}$ with $r_i > 0$ for
all $i$. The vertices of $Q_{n,k}$ can be identified with pairs $(i,B)$ with
$i \in [n]$ and $B \subseteq [k]$. There is a directed edge from $(i,B)$ to
$(j,C)$ if and only if $i < j$ and $B=C$ or $i=j$ and $C = B \cup \{b\}$ for
some $b \in [k] \setminus B$. In particular the unique sink is $(n,[k])$. 

\renewcommand\AA{\mathbb{A}}%
\begin{thm}\label{thm:multi_multiplihedron}
    Let $(\AA,\cdot)$ be a non-associative monoid and let $f_1,\dots, f_k :
    \AA \to \AA$ be morphisms. The vertices of $\PP(Q_{n,k},(c,r))$ are in
    bijection with the possible ways of evaluating 
    \[
        (f_{\sigma(1)} \circ f_{\sigma(2)} \circ \cdots \circ
        f_{\sigma(k)})(a_1 \cdot a_2 \cdots a_n)\, ,
    \]
    where $a_1,\dots,a_n \in \AA$ and $\sigma \in \SymGrp_k$ is a permutation.
\end{thm}

For $k \ge 2$, all polytopes $\PP(Q_{n,k},(c,r))$ of dimension $\le 3$ are
permutahedra. \Cref{fig:multi} gives an illustration for $n=k=2$.  Note that
if $k=0$, then $Q_{n,0} = \Simp[n-1]$ and $\PP(Q_{n,k},(c,r))$ is the
associahedron by \Cref{thm:main_asso}. For $k=1$,
\Cref{thm:multi_multiplihedron} yields that the vertices of $\PP(\Simp[n-1]
\times \Simp[1],(c,r))$ are in bijection with the vertices of Stasheff's
\Def{multiplihedron}; see \cite{Stasheff-book}. Indeed, $\PP(\Simp[n-1] \times
\Simp[1],(c,r))$ is the constrainahedron $C(2,n) = C(n,2)$ by
\Cref{thm:constrainahedron}, which coincides with the multiplihedron
by~\cite[Theorem 2.1]{BottmanPoliakova}. \Cref{thm:multi_multiplihedron}
insinuates that $\PP(Q_{n,k},(c,r))$ is combinatorially isomorphic to the
$(k,n)$-multiplihedron of Chapoton--Pilaud~\cite{ChapotonPilaud}.  Germain
Poullot found a general piecewise-linear connection between max-slope pivot
rule polytopes of products of simplices and shuffle products of associahedra;
see~\cite[Section~3.2]{poullot:tel-04269354}.

\begin{figure}[h]
    \begin{center}
        \includegraphics[height=2.5cm]{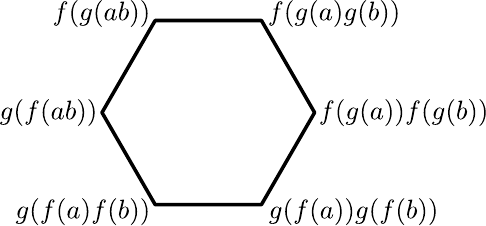}
    \end{center}
    \caption{An illustration of the $(2,2)$-multiplihedron.}
    \label{fig:multi}
\end{figure}

Consider an evaluation of $(f_{\sigma(1)} \circ f_{\sigma(2)} \circ \cdots \circ
f_{\sigma(k)})(a_1 \cdot a_2 \cdots a_n)$ with $\sigma$ fixed. Disregarding the morphisms for a
moment, the order in which the $n-1$ multiplications have to be carried out is
encoded by a partial order $\preceq_T$ on $[n-1]$ whose Hasse diagram is
well-known to be a tree $T$ rooted at the maximum (i.e., the last multiplication
to be carried out). Every node in the tree has at most two children. With the
labelling given by $[n-1]$, the possible posets are precisely the binary search
trees on $n-1$ nodes. Now, for the $j$-th multiplication, we record the number
$\phi(j)$ of morphisms that have to be applied before the multiplication can
take place. The poset $([n-1],\preceq_T)$ together with the map $\phi : [n-1] \to
\{0,1,\dots,k\}$ completely determines the evaluation. Note that if $i \prec j$,
then $\phi(i) \le \phi(j)$ and hence $\phi$ is an \Def{order-preserving map}.

Let $(\wt,s) \in \R^{n+k}$ be a generic weight. As before, we define
$\tau(i,j) := \frac{\wt_j - \wt_i}{c_j - c_i}$ for $1 \le i < j \le n$ and
$\tau(i) = \max_{j >i} \tau(i,j)$. By genericity we can assume that $s_g \neq
s_h$ for $g \neq h$. Let $(i,B) \neq (n,[k])$ and let $h \in [k] \setminus B$
such that $s_h$ is maximal. Unravelling \eqref{eqn:max_slope} we find that 
\[
    \arb^{(\wt,s)}(i,B) \ = \
    \begin{cases} 
        (\arb^\wt(i),B) & \text{ if } \tau(i) > s_h \\
        (i,B \cup \{h\}) & \text{ if } \tau(i) < s_h.
    \end{cases}
\]
There is unique permutation $\sigma \in \SymGrp_k$ such that $s_{\sigma(1)} >
s_{\sigma(2)} > \cdots > s_{\sigma(k)}$. Moreover, for every $i \in [n-1]$,
define $\phi_s(i) := |\{h \in [k] : \tau(i) > s_h \}|$. Then the arborescence
$\arb^{\wt,s}$ is completely determined by $(\tau,\sigma,\phi_s)$.

Recall from \Cref{lem:poset} that any noncrossing arborescence $\arb : [n-1]
\to [n]$ determines a partial order $\preceq$ on $[n-1]$ whose Hasse diagram
is given by a binary search tree on $[n-1]$. 

\begin{lem}\label{lem:slope_poset}
    For a generic weight $\wt \in \R^n$ let $([n-1],\preceq)$ be the
    poset associated to the noncrossing arborescence $\arb^\wt$. If $a \prec
    b$, then $\tau(a) > \tau(b)$.
\end{lem}
\begin{proof}
    It suffices to assume that $a$ is covered by $b$, that is, $a \prec b$ and
    there is no $s$ such that $a \prec s \prec b$. It follows from
    \Cref{lem:poset} that this is the case if $a = b-1$ and $\arb^\wt(a) = b$
    or $a = b+1$ and $\arb^\wt(a) = \arb^\wt(b) = a+1$.  

    Let $c = \arb^\wt(b)$ In the former case, \Cref{lem:conv} with $(r,s,t) =
    (a,b,c)$ implies $\tau(a) = \tau(a,b) > \tau(a,c) > \tau(b,c) = \tau(b)$.
    In the latter case when $b < a$, $\tau(b) = \tau(b,c) > \tau(b,a)$ and
    \Cref{lem:conv} with $(r,s,t) = (b,a,c)$ yields $\tau(b,c) < \tau(a,c) =
    \tau(a)$.
\end{proof}

\begin{proof}[Proof of \Cref{thm:multi_multiplihedron}]
    Let $(\wt,s)$ be a generic weight. The noncrossing arborescence
    $\arb^\wt$ determines a partial order $([n-1],\preceq)$ whose Hasse
    diagram is a binary search tree $T$. By \Cref{lem:slope_poset}, the map
    $\phi_s$ associated to $s$ is order-preserving and takes values in
    $\{0,1\dots,k\}$. Thus, together with the permutation $\sigma$, this
    uniquely determines a unique evaluation of $(f_{\sigma(1)} \circ
    f_{\sigma(2)} \circ \cdots \circ f_{\sigma(k)})(a_1 \cdot a_2 \cdots a_n)$
    as discussed above. Conversely, given such an evaluation represented by a
    permutation $\sigma$, a partial order $([n-1],\preceq)$, and an
    order-preserving map $\phi$, there is a unique noncrossing arborescence
    $\arb$ which gives rise to $([n-1],\preceq)$. By \Cref{thm:noncross}, we
    can find $\wt \in \R^n$ with $\arb = \arb^\wt$. Finally, given $\sigma$ and
    $\phi$, we can find numbers $s_1,\dots,s_k$ with $s_{\sigma(i)} >
    s_{\sigma(i+1)}$ and $\phi(i) = \phi_s(i)$ for all $i=1,\dots,k$. 
\end{proof}

The \Def{order polynomial} $\Omega_P(l)$ of a partially ordered set
$(P,\preceq)$ is a polynomial of degree $|P|$ such that $\Omega_P(l)$ is the
number of order-preserving maps $P \to [l]$. The order polynomial was
introduced by Stanley and is well-studied; see, for example,~\cite{EC1,CRT}.
For fixed $n \ge 2$, we define 
\[
    V_{n}(k) \ := \ \sum_T \Omega_T(k+1) \, ,
\]
where the sum is over all binary search trees $T$ of on $n$ nodes. Our encoding of
an evaluation of $(f_{\sigma(1)} \circ f_{\sigma(2)} \circ \cdots \circ
f_{\sigma(k)})(a_1 \cdot a_2 \cdots a_n)$ by a binary search tree and an
order-preserving map into $\{0,1,\dots,k\}$ yields the following.

\begin{cor}\label{cor:n_vertices}
    The number of evaluations of $(f_{\sigma(1)} \circ f_{\sigma(2)} \circ
    \cdots \circ f_{\sigma(k)})(a_1 \cdot a_2 \cdots a_n)$ is $V_n(k)$. In
    particular, the number of vertices of $\PP(Q_{n,k},(c,r))$ is $k!V_n(k)$.
\end{cor}

The number of vertices of the $(k,n)$-multiplihedron is given Proposition 126
of~\cite{ChapotonPilaud} as $k!$ times the coefficient of $y^{n+1}$ of the
power series $C^{(m+1)}(y)$, where $C^{(1)}(y) = \frac{1 - \sqrt{1-4y}}{2}$
is the Catalan generating function and $C^{(i+1)}(y) = C(C^{(i)}(y))$. It is
remarkable that the number of vertices for varying $k$ is essentially given by
a polynomial. The following table shows a few of the polynomials $V_{n}(k)$. 
\[
    \def\arraystretch{1.2}
    \begin{array}{l|l}
        n & V_{n}(k)\\
        \hline
        2 &  k^{2} + 3 k + 2\\
3 &  k^{3} + \frac{11}{2} k^{2} + \frac{19}{2} k + 5\\
4 &  k^{4} + \frac{25}{3} k^{3} + 25 k^{2} + \frac{95}{3} k + 14\\
5 &  k^{5} + \frac{137}{12} k^{4} + \frac{101}{2} k^{3} + \frac{1291}{12} k^{2} + \frac{219}{2} k + 42\\
6 &  k^{6} + \frac{147}{10} k^{5} + \frac{263}{3} k^{4} + \frac{541}{2} k^{3} + \frac{1360}{3} k^{2} + \frac{1944}{5} k + 132\\
7 &  k^{7} + \frac{363}{20} k^{6} + \frac{2069}{15} k^{5} + \frac{6809}{12} k^{4} + \frac{8161}{6} k^{3} + \frac{56773}{30} k^{2} + \frac{14079}{10} k + 429\\
8 &  k^{8} + \frac{761}{35} k^{7} + \frac{36461}{180} k^{6} + \frac{63229}{60} k^{5} + \frac{60125}{18} k^{4} + \frac{395189}{60} k^{3} + \frac{1415369}{180} k^{2} + \frac{362247}{70} k + 1430\\
9 &  k^{9} + \frac{7129}{280} k^{8} + \frac{7915}{28} k^{7} + \frac{64549}{36} k^{6} + \frac{42877}{6} k^{5} + \frac{6665857}{360} k^{4} + \frac{373321}{12} k^{3} + \frac{8211733}{252} k^{2} + \frac{269403}{14} k + 4862\\
    \end{array}
\]

As the table shows, all coefficients are non-negative and, of course, we
tested if the polynomials are real-rooted. This seems to be true for $n \le
10$ but unknown beyond. For $n \le 10$, $(n-1)!V_{n}(k)$ has integer
coefficients which are log-concave.  We end with two simple observations
regarding the polynomial.

\begin{thm}
    For $n \ge 2$, the leading coefficient of $V_n(k)$ is $1$ and the constant
    coefficient is the $n$-th Catalan number $C_n$.
\end{thm}
\begin{proof}
    For the constant coefficient, we simply note that $\Omega_T(1) = 1$ and
    hence $V_{n}(0)$ is the number of trees $C_n$. For the leading
    coefficient, we recall that the leading coefficient of $n!\Omega_T(l)$ is
    the number of linear extensions $e(T)$ of $T$, that is, order-preserving
    bijections $\ell : T \to [n]$; see~\cite[Sect.~6.2]{CRT}.  If $\ell =
    (\ell_1,\dots,\ell_n)$ is a list of distinct numbers, there is a simple
    recursive procedure that produces a binary tree $T$ on nodes $1,\dots,n$
    such that $i \mapsto \ell_i$ gives a order preserving map on $T$;
    see, for example,~\cite{tonks}. If $i$
    is the index such that $\ell_i$ is maximal, then set $i$ to be the root of
    the binary tree. The left subtree is determined by
    $(\ell_1,\dots,\ell_{i-1})$ and the right subtree is determined by
    $(\ell_{i+1},\dots,\ell_n)$. Thus, every bijection $[n] \to [n]$
    determines a unique binary tree $T$ and the number of bijections that
    produce $T$ is $e(T)$. Hence $\sum_{T}e(T) = n!$.
\end{proof}

\bibliographystyle{siam}
\bibliography{References.bib} 
\end{document}